\documentclass[11pt,letterpaper]{amsart}

\usepackage{silence}
\usepackage{subfiles}
\usepackage[T1]{fontenc}
\usepackage[utf8]{inputenc}
\usepackage[dvipsnames]{xcolor}
\usepackage[yyyymmdd,hhmmss]{datetime}
\usepackage[numbers,sort,compress]{natbib}
\usepackage{amstext,amsthm,amssymb}
\usepackage{xparse,mathrsfs,mathtools}
\usepackage[british]{babel}
\usepackage[babel=true,english=british]{csquotes}
\usepackage{xr-hyper}
\usepackage{stmaryrd}
\usepackage{enumitem}
\usepackage{bbm, dsfont}
\usepackage[pdftex,unicode=true,%
pdfusetitle,pdfdisplaydoctitle=true,%
pdfpagemode=UseOutlines,%
bookmarks=true,bookmarksnumbered=false,bookmarksopen=true,bookmarksopenlevel=2,%
breaklinks=true,%
pdfencoding=unicode,psdextra,
pdfcreator={},
pdfborder={0 0 0},
colorlinks=false
]{hyperref}
\usepackage{lmodern,microtype}
\usepackage{geometry}
\usepackage[title]{appendix}
\usepackage[capitalise,nameinlink]{cleveref}
\usepackage[cal=zapfc,bb=ams,frak=esstix,scr=boondoxo]{mathalfa}

\hypersetup{
	colorlinks=false,
	pdfborder={0 0 0},
	pdfborderstyle={/S/U/W 0},
}

\newcommand*{\doi}[1]{\href{http://dx.doi.org/#1}{doi: {\texttt{\footnotesize\color{teal}#1}}}}

\newtheorem*{thm*}{Theorem}
\newtheorem{thm}{Theorem}[section]
\newtheorem{lem}[thm]{Lemma}

\newtheorem{prop}[thm]{Proposition}

\newtheorem{Def}[thm]{Definition}

\theoremstyle{remark}

\newtheorem*{rem*}{Remark}
\theoremstyle{plain}
\newcommand{\thistheoremname}{}
\newtheorem*{genericthm*}{\thistheoremname}
\newenvironment{namedthm*}[1]
{\renewcommand{\thistheoremname}{#1}%
	\begin{genericthm*}}
	{\end{genericthm*}}
\numberwithin{equation}{section}
\crefformat{equation}{#2(#1)#3}
\crefformat{enumi}{#2(#1)#3}
\crefname{section}{\textsection}{\textsection\textsection}
\DeclarePairedDelimiter\parentheses{\lparen}{\rparen}
\DeclarePairedDelimiter\braces{\lbrace}{\rbrace}
\DeclarePairedDelimiter\brackets{\lbrack}{\rbrack}
\DeclarePairedDelimiter\abs{\lvert}{\rvert}
\DeclarePairedDelimiter\norm{\lVert}{\rVert}

\NewDocumentCommand\set{ s o m o }{%
	\IfBooleanTF{#1}{\IfNoValueTF{#4}{\braces*{#3}}{\braces*{\,#3:#4\,}}}{%
		\IfNoValueTF{#2}{\IfNoValueTF{#4}{\braces{#3}}{\braces{\,#3:#4\,}}}{%
			\IfNoValueTF{#4}{\braces[#2]{#3}}{\braces[#2]{\,#3:#4\,}}}}%
}
\NewDocumentCommand\e{ s O{} m }{%
	\IfBooleanTF{#1}{%
		\operatorname{e}_{#2}\parentheses*{#3}%
	}{\operatorname{e}_{#2}\parentheses{#3}}%
}

\DeclareSymbolFont{symbolsC}{U}{txsyc}{m}{n}
\DeclareMathSymbol{\notniFromTxfonts}{\mathrel}{symbolsC}{61}
\makeatother

\title[Joint extreme values]{Joint extreme values of $L$-functions on and off the critical line}

\subjclass[2020]{% See http://www.ams.org/msc/msc2020.html
	11M06, 11F66}


\author{Athanasios~Sourmelidis}
\address{%
	Univ. Lille, CNRS\\%
	UMR 8524 -- Laboratoire Paul Painlevé\\%
	F-59000~Lille\\%
	France%
}
\email{athanasios.sourmelidis@univ-lille.fr}

\begin{document}
	\begin{abstract}
		It is shown that any number of distinct primitive $\mathrm{GL}(1)$ and $\mathrm{GL}(2)$ $L$-functions can simultaneously attain large values on the critical line.
		This is an unconditional improvement of a general result due to Heap and Li who have assumed the Riemann Hypothesis for more than three such $L$-functions.
		The joint distribution of $\mathrm{GL}(m)$ $L$-functions  to the right of the critical line  is also studied under certain zero-density estimates.
		In particular, we can partially  recover results of Inoue and Li on Dirichlet $L$-functions and generally improve upon the work of Mahatab, Pa\'nkowski and Vatwani on the class of $L$-functions introduced by Selberg.
		The main machinery in both cases, on and off the critical line, is the resonance method of Soundararajan and Hilberdink/Voronin, respectively.
		On the critical line we additionally introduce a variation of Heath-Brown's method for the fractional moments of the Riemann zeta-function which makes it possible to avoid using any information on the zero distribution of $L$-functions whose degree is less than three.
	\end{abstract}
	\maketitle
	
	\section{Introduction and Main results}
	The study of extreme values for zeta- and $L$-functions dates back to the work of Bohr, Landau, Littlewood and Titchmarsh \cite[Chapter VIII]{Titchmarsh1986}.
	It had several renaissances over the past century, most notably through the work of Selberg \cite{Selberg1946} for the argument of the Riemann zeta-function $\zeta(s)$, $s:=\sigma+it\in\mathbb{C}$, on the {\it critical line} $1/2+i\mathbb{R}$, the work of Levinson \cite{Levinson1972} and Montgomery \cite{Montgomery1977} for $\zeta(s)$ to the right of the critical line and the work of Balasubramanian and Ramachandra \cite{Balasubramanian1977} for $\zeta(1/2+it)$.
	The latest major addition to the mathematical toolkit that is utilized for addressing such problems has been provided by Soundararajan \cite{Soundararajan2008} who introduced  the so-called {\it resonance method}.
	Apart from improving upon \cite{Balasubramanian1977}, the versatility of this technique allowed new developments into the subject of extreme values of $L$-functions in the {\it $t$-aspect} (see \cite{Aistleitner2015,Bondarenko2017,Breteche2018} for some benchmarks) and in the {\it family/conductor aspect} (see \cite{Blomer2023} for a comprehensive treatment of the topic).
	
	In recent years attention has also turned to joint/simultaneously extreme values (in both aspects) which are contained in the more general study of the joint value distribution of $L$-functions having Dirichlet coefficients of number-theoretic interest. 
	For a family of such functions $L_h(s)=\sum_{n\geq1}a_{L_h}(s)n^{-s}$, $h\leq H$, there is a guiding principle that allows some degree of intuition on the results one can expect in such cases and it consists of three parts.
	The $L$-functions should at first have an {\it analytic continuation}, modulo finitely many poles, beyond their half plane of absolute convergence since  it is there that their value distribution yields important information on the arithmetic properties of $a_{L_h}(n)$.
	Secondly, it is desirable for the $L$-functions to have an {\it Euler product} representation, for then they can, loosely speaking, be expressed as
	\[
	L_h(s)\approx\exp\parentheses*{\sum_{p}a_{L_h}(p)p^{-s}}\times\parentheses*{\text{contribution from zeros}}.
	\]
	If the contribution from the zeros of $L_h(s)$ (the analytic factor of the above expression) is benign, then the sum over primes $p$ (the arithmetic factor) determines the behavior of $L_h(s)$.
	This can be very practical, though challenging to establish unconditionally, because the numbers $\log p$, prime $p$, are linearly independent over $\mathbb{Q}$.
	Therefore, the functions $p^{-it}$, prime $p$, behave like i.i.d. random variables and the above model of $L_h$ can be replaced by a random model which not only provides heuristics but it is also less rigid to handle.
	The last prerequisite for studying the joint distribution of $L_h$, $h\leq H$, is for them to exhibit some sort of {\it independence on arithmetic level}, namely that
	\begin{align}\label{PNT0}
	\frac{\log x}{x}\sum_{p\leq x}a_{L_h}(p)\overline{a_{L_g}(p)}\mathop{\longrightarrow}^{x\to+\infty}\kappa(h,g),
	\end{align}
	where $\kappa(h,g)\geq0$ and is non-zero only when $h=g$.
	This condition ensures that the entries of the tuple $(L_1,\dots, L_H)$ are uncorrelated  and act independently in the complex plane.
	In contemplation, this indicates that the state of the art on the value-distribution for a single $L$-function should also be within reach for a tuple of independent $L$-functions.
	
	An example is given by a family of Dirichlet $L$-functions $L(s,\chi_h):=\sum_{n\geq1}\chi_h(n)n^{-s}$, attached to nonequivalent  Dirichlet characters $\chi_h$, $h\leq H$, in which case \eqref{PNT0} follows from the orthogonality of characters and the prime number theorem.
	Selberg \cite{Selbergclass} described how his celebrated central limit theorem for $\log\zeta(s)$ could be adjusted to show that the functions
	\[
	\frac{\log\abs{L\parentheses{\frac{1}{2}+it,\chi_1}}}{\sqrt{\pi\log\log t}},\, 	\frac{\arg{L\parentheses{\frac{1}{2}+it,\chi_1}}}{\sqrt{\pi\log\log t}},\cdots,	\frac{\log\abs{L\parentheses{\frac{1}{2}+it,\chi_H}}}{\sqrt{\pi\log\log t}},\, 	\frac{\arg{L_H\parentheses{\frac{1}{2}+it,\chi_H}}}{\sqrt{\pi\log\log t}}
	\]
	become distributed, in the limit range $t$, like independent random variables, each having Gaussian density $\exp(-\pi u^2)\mathrm{d}u$.
Actually Selberg stated this in a more general setting, where he considers a class $\mathcal{S}$ of $L$-functions  (see Section \ref{classesofL} for the precise definitions) whose {\it primitive} elements satisfy, among other things, the three requirements of the principle previously mentioned.
A more systematic treatment of the value-distribution of $L$-functions from a subclass of $\mathcal{S}$ and a rigorous proof of the Gaussian-like behavior of their logarithms on the critical line has been carried out by Bombieri and Hejhal \cite{Bombieri1995}.
In general, the aforementioned classes and their variants, as they appear in the literature, are built by having in mind $L$-functions associated with automorphic representations of  $\mathrm{GL}(m)$ over $\mathbb{Q}$ and, at least in the case of $\mathcal{S}$, it is conjectured that these $L$-functions are its only elements.

	In the context of extreme values we would expect that each entry of a tuple $(L_1,\dots, L_H)$ of independent $L$-functions can be large in magnitude (as depicted in \cite{Aistleitner2017} for example) simultaneously.
	A first result was obtained by Mahatab, Pa\'nkowski and Vatwani \cite{Mahatab2022} in the case when $L_h\in\mathcal{S}_{\text poly}$, $h\leq H$, where $\mathcal{S}_{\text{poly}}\subseteq\mathcal{S}$ consists of the elements from $\mathcal{S}$ that can be represented by a {\it polynomial Euler product}. 
	They showed that if these $L$-functions do not have too many zeros off the critical line and satisfy the {Ramanujan Hypothesis} and a version of \eqref{PNT0}, then for any fixed $\sigma\in(1/2,1)$ and $\phi_1,\dots,\phi_H\in\mathbb{R}$, and for any sufficiently large $T\gg1$, there are $t_1,\dots, t_H\in[T,2T]$ such that
		\begin{align}\label{motivation0}
	\min_{h\leq H}	\mathrm{Re}\parentheses{e^{-i\phi_h}\log L_h(\sigma+it_h)}\gg\frac{(\log T)^{1-\sigma}}{\log\log T},
	\end{align}
with $|t_h-t_g|\leq2(\log T)^{(1+\sigma)/2}\sqrt{\log\log T}$ for all $g,h\in\set{1,\dots,H}$.
	The statement about the zeros is given with respect to the order of growth of the zero-counting functions in rectangles of the $L$-functions, namely that
	\[
	N_{L_h}(\sigma,T):=\#\set*{\rho:L_h(\rho)=0,\,T\leq \mathrm{Im}\rho\leq 2T,\,\mathrm{Re}\rho\geq\sigma},\quad h\leq H,\,\sigma>\frac{1}{2},\,T\geq1,
	\]
	grow moderately slow as $T\to+\infty$.
	The above result covers the case of Dirichlet $L$-functions and $L$-functions associated with holomorphic cusp forms of $\mathrm{GL}(2)$ over $\mathbb{Q}$  in the whole range $(1/2,1)$.
	Moreover, it covers  the case of higher degree $\mathrm{GL}(m)$ $L$-functions in a restricted neighborhood of $1$, depending on $m$, where it is known that $	N_{L}(\sigma,T)$ grows sufficiently slow.
	Although \eqref{motivation0} yields joint extreme values in directions (including small and large values) for a wide class of $L$-functions, it does not provide us with a single $t$ that works for all entries of $(L_1,\dots,L_H)$ simultaneously but rather produces a collection of $t_h$, $h\leq H$, which are not too far from each other.
	This has been resolved  by Inoue and Li \cite[Theorem 1]{Inoue2025}, who showed that, for Dirichlet $L$-functions, \eqref{motivation0} holds with $t_h=t$ for some $t\in[T,2T]$ and all $h\leq H$ (they actually prove more). 
	
	On the critical line, Heap and Li \cite[Theorem 1]{Heap2024} proved that if $\mathcal{S}_0\subseteq\mathcal{S}_{\text{poly}}$ is the class of $L$-functions consisting of the Riemann zeta-function, the primitive Dirichlet $L$-functions and the $L$-functions associated with primitive cusp forms on $GL(2)$ over $\mathbb{Q}$, then for $H=2$, any distinct $L_h\in\mathcal{S}_0$, $h\leq H$, and any sufficiently large $T\gg1$ we have that
		\begin{align}\label{motibvation1}
	\max_{t\in[T,2T]}\min_{h\leq H}|L_h({\textstyle{\frac{1}{2}}}+it)|\geq\exp\parentheses*{\sqrt{\frac{D\log T}{H\log\log T}}},
	\end{align}
	where $D>0$ is a sufficiently  small constant.
They also describe how their method can be modified to obtain the above inequality when $H=3$ and all $L_h$ are Dirichlet $L$-functions.
For $H>3$ however, they  have to assume the {\it Generalized Riemann Hypothesis} (GRH) in which case any positive $D<{1}/{2}$ is admissible \cite[Theorem 3]{Heap2024}.
The first objective here is to make the latter result unconditional.
	\begin{thm}\label{thm:ourresult0}
		If $L_1,\dots, L_H$ are elements of $\mathcal{S}_0$, then for any sufficiently small but fixed $D>0$ and for any sufficiently large $T\gg1$ we have that
		\[
		\max_{t\in[T,2T]}\min_{h\leq H}|L_h({\textstyle{\frac{1}{2}}}+it)|\geq\exp\parentheses*{\sqrt{\frac{D\log T}{H(H-1)\log\log T}}}.
		\]
		If every $L_h$, $h\leq H$, is a Dirichlet $L$-function, then any $D<{1}/{2}$ is admissible. 
		Otherwise, we can take $D<\frac{ 1/2-\theta}{3+\theta}$ where $\theta$ is an admissible bound towards the Ramanujan Hypothesis for cusp forms on $G L(2)$ over $\mathbb{Q}$.
	\end{thm}
	Leaving aside that the admissible range of $D$ is  more restricted compared to \cite[Theorem 1, Theorem 3]{Heap2024}, the lower bound of Theorem \ref{thm:ourresult0} is in addition smaller by a power of $1/\sqrt{H-1}$ than the one presented in \eqref{motibvation1}.
	Interestingly, if we were to also assume GRH then we could remove this power.
	In fact we could recover \cite[Theorem 3]{Heap2024}, which conditionally treats any number of $\mathrm{GL}(m)$ $L$-functions, under the same hypotheses.
To prove Theorem \ref{thm:ourresult0} we employ the resonance method of Soundararajan  and a variation of this technique will also 
be used to obtain results on extreme values (both large and small) off the critical line. 
	
	\begin{thm}\label{thm:ourresult1}
		Let $L_1,\dots,L_{H+G}$ be elements of $\mathcal{S}_{\text{poly}}$ satisfying the Ramanujan Hypothesis and \eqref{PNT0}.
		If there are  $\sigma_0\in(1/2,1)$ and $\alpha\in(0,1)$ with  $\sum_{h\leq H+G}N_{L_h}(\sigma_0,T)=o(T^{1-\alpha})$, then there is $D>0$ such that for any $\sigma\in(\sigma_0,1)$ and any sufficiently large  $T\gg1$
		\[
		\max_{t\in[T,2T]}\min\braces*{\min_{h\leq H}|L_h(\sigma+it)|,\max_{H<h\leq H+G}\abs*{L_{h}(\sigma+it)}^{-1}}\geq\exp\parentheses*{D\frac{(\log T)^{1-\sigma}}{\log\log T}}.
		\]
		For $\sigma=1$, the right-hand side of the above inequality ought to be replaced by $(\log\log T)^{D}$.
	\end{thm}
	It is unfortunate that we have to impose the Ramanujan Hypothesis, for we can not cover unconditionally the case of $L$-functions associated with Maaß forms.
	We can circumvent this restriction in Theorem \ref{thm:ourresult0} because the Dirichlet polynomial utilized in the resonance method on the critical line is built by many more and much larger primes compared to the one used off the critical line. 
	There is hence some flexibility on how to choose its Dirichlet coefficients in the former case, following an idea of Heap and Li, without waiving any of its defining properties.
Although the assumption on zero-density estimates is necessitated to be able to represent the $L$-functions by short Dirichlet polynomials, the situation off the critical line is simpler and does not require the stronger assumption of the GRH.
Moreover, such zero-density estimates are already known for $L$-functions from $\mathcal{S}_0$, while similar estimates for higher degree $\mathrm{GL}(m)$ $L$-functions close to the critical line are well beyond out of reach with the current technology.
	All things considered,  Theorem \ref{thm:ourresult1} recovers partially \cite[Theorem 1]{Inoue2025}, covers unconditionally the case of $L$-functions associated with holomorphic cusp forms on $\mathrm{GL}(2)$ over $\mathbb{Q}$ and, under weaker assumptions, improves upon \cite{Mahatab2022} by showing that \eqref{motivation0} holds with $t_h=t$ for some $t\in[T,2T]$ and all $h\leq H$ when $\phi_h\equiv0\bmod \pi$.
	Incidentally, a careful review of the proof of \eqref{motivation0} shows that if one is willing to not enforce the $t_h$, $h\leq H$, to be identical, then the 
	$L$-functions can be arbitrarily chosen and need not be independent.
	
	\begin{thm}\label{thm:ourresult2}
		Let $L_1,\dots,L_H\in\mathcal{S}$.
		Assume that there are $\sigma_0\in(1/2,1)$ and $\alpha\in(0,1)$ such that  $\sum_{h\leq H}N_{L_h}(\sigma_0,T)=o(T^{1-\alpha})$. 
		Assume also that there are $\kappa_h>0$, $h\leq H+G$, with
		\begin{align}\label{PNT1}
			\frac{\log x}{x}\sum_{p\leq x}|a_{L_h}(p)|^{\mathrm{i}(h)}\mathop{\longrightarrow}^{x\to+\infty}\kappa_h,
		\end{align}
	where $\mathrm{i}(h)\in\set{1,2}$ if $L_h\in\mathcal{S}_{\text{poly}}$ and satisfies the Ramanujan Hypothesis and $\mathrm{i}(h)=1$ otherwise.
		For any $\phi_1,\dots,\phi_H\in\mathbb{R}$ and $T\gg1$, there are $t_1,\dots,t_H\in[T,2T]$ such that
		\[
		\min_{h\leq H}	\mathrm{Re}\parentheses{e^{-i\phi_h}\log L_h(\sigma_0+it_h)}\gg\frac{(\log T)^{1-\sigma_0}}{\log\log T},
		\]
		with $|t_h-t_g|\leq2(\log T)^{(1+\sigma_0)/2}\sqrt{\log\log T}$ for all $g,h\in\set{1,\dots,H}$.
	\end{thm}
		\paragraph{\bf Structure and Notations}
		The next two sections  are dedicated in an outline of the core ideas behind the proofs of the main results and the introduction of the classes $\mathcal{S}_{\text{poly}}\subseteq\mathcal{S}$ along with short descriptions of the distinctive properties that their elements share.
	The proof of Theorem \ref{thm:ourresult0} is presented 	in Section \ref{firsttheorem} by assuming the truth of Proposition \ref{keyprop}.
		This is the key result of the current work and Section \ref{seckeyprop} will be devoted entirely on how to obtain it, since its statement could be of independent interest and its proof is lengthy.
Theorems \ref{thm:ourresult1} and \ref{thm:ourresult2} will be proved in Section \ref{secondthm} and \ref{thirdthm}, respectively.
In Section \ref{concluding} we make an informal discussion on conditional aspects of Theorem \ref{thm:ourresult0} and unconditional applications of Theorem \ref{thm:ourresult1} and Theorem \ref{thm:ourresult2}.
An appendix is provided in the end containing a list of technical yet necessary estimates for the proof of Theorem \ref{thm:ourresult0}.
		
	We use the Landau notation $f(x)=O(g(x)$ and the Vinogradov notation $f(x)\ll g(x)$ to mean that there exists some constant $C>0$ such that $|f(x)|\leq Cg(x)$ holds for all admissible values of $x$ (where the meaning of ``admissible'' will be clear from the context). 
	If both $f$ and $g$ are positive functions, then $f(x)\asymp g(x)$ will stand for $f(x)\ll g(x)\ll f(x)$.
	We write $f(x)=o(g(x))$ if $g(x)$ is positive for all sufficiently large values of $x$ and $f(x)/g(x)$ tends to zero as $x\to+\infty$.
	Unless stated otherwise, the notation ``$o(1)$'' is reserved for a function of the variable $T$ that converges to $0$ as $T\to+\infty$.
	The parameter $\epsilon$ will denote throughout the article a sufficiently small positive number which may differ from line to line but can be fixed eventually.	
	All implicit constants are absolute with respect to $T$ but may depend on the occasional parameters within the proofs.
	Lastly, the notation $(a,b)$ will stand for the greatest common divisor of integers $a$ and $b$. 
	\section{An overview of the ideas}
	\subsection{On the critical line}
Let $L_h\in\mathcal{S}_0$, $h\leq H$, and $L=\prod_{h\leq H}L_h$ be their product.
	Heap and Li \cite{Heap2024} observe that \eqref{motibvation1} follows if we can show that there is $t\in[T,2T]$ such that
		\begin{align}\label{mainidea0}
		\abs{L\parentheses{{\textstyle{\frac{1}{2}+it}}}}^{q}\geq V^q\sum_{h\leq H}\prod_{g\neq h}\abs{L_g\parentheses{{\textstyle{\frac{1}{2}+it}}}}^{q},\quad V=\exp\parentheses*{(D+o(1))\sqrt{\frac{\log T}{\log\log T}}},
	\end{align}
	where $q=2$ and $D$ is some suitably chosen positive constant.
	A sufficient condition to establish the latter inequality would then be to show that, for every $h\leq H$,
		\begin{align}\label{mainidea}
		\int_{T}^{2T}\abs{L\parentheses{{\textstyle{\frac{1}{2}+it}}}}^{q}\abs{R\parentheses{{\textstyle{\frac{1}{2}+it}}}}^{2}\mathrm{d}t\gg {V}^q\int_{T}^{2T}\abs{R\parentheses{{\textstyle{\frac{1}{2}+it}}}}^{2}\prod_{g\neq h}\abs{L_g\parentheses{{\textstyle{\frac{1}{2}+it}}}}^{q}\mathrm{d}t,
	\end{align}
	with any possible discrepancies on the size of the implicit constants being absorbed by the factor $\exp\parentheses*{o(1)\sqrt{\frac{\log T}{\log\log T}}}$ of $V$.    
Thus, it all boils down in obtaining a sharp lower bound for the left-hand side of \eqref{mainidea} and a sharp upper bound for its right-hand side.
	The lower bound is simpler to derive because after invoking H\"older's inequality to show that
		\begin{align}\label{Hoelderlimit}	\int_{T}^{2T}\abs{L\parentheses{{\textstyle{\frac{1}{2}+it}}}}^{q}\abs{R\parentheses{{\textstyle{\frac{1}{2}+it}}}}^{2}\mathrm{d}t\geq\abs*{\int_T^{2T}{L\parentheses{{\textstyle{\frac{1}{2}+it}}}}\abs{R\parentheses{{\textstyle{\frac{1}{2}+it}}}}^{2}\mathrm{d}t}^{q}\parentheses*{\int_T^{2T}\abs{R\parentheses{{\textstyle{\frac{1}{2}+it}}}}^{2}\mathrm{d}t}^{1-q},
		\end{align}
		it remains to bound from below the first factor on the right-hand side and from above the second factor.
		But the former follows from a standard procedure (see \cite[Section 2]{Aistleitner2017} or \cite[Section 4]{Heap2024}), while the latter is designed from the  beginning to be easily established.
		
		The task of obtaining sharp upper bounds for the right-hand side of \eqref{mainidea} is more difficult and stems out from understanding the same moments untwisted by $|R(\frac{1}{2}+it)|^2$.
		When $H=2$ and $H=3$, these moments reduce to
	\[
\int_{T}^{2T}|L_h({\textstyle{\frac{1}{2}}}+it)|^q\mathrm{d}t\quad\text{and}\quad\int_{T}^{2T}|L_h({\textstyle{\frac{1}{2}}}+it)L_g({\textstyle{\frac{1}{2}}}+it)|^q\mathrm{d}t,\quad L_h,L_g\in\mathcal{S}_0,
\] 
respectively.
In the first case we have  precise asymptotics (recall we set $q=2$) and the resonance method is applicable as has been shown in \cite[Proposition 3]{Heap2024}.
We already encounter difficulties when $H=3$ because the only moments
for which we have sharp upper bounds are if both $L_h$ and $L_g$ are Dirichlet $L$-functions.
When $H>3$ nothing is known unless we assume the GRH  (all {\it nontrivial} zeros of the involved $L$-functions lie on the critical line),  in which case Hagen \cite{Hagen2025} recently provided satisfying answers by building on the work of Soundararajan \cite{Soundararajan2009} and Harper \cite{Harper2013}.
In that regard, Heap and Li \cite[Proposition 4]{Heap2024} incorporate ideas of Soundararajan and Harper to  conditionally obtain sufficiently good upper bounds for the right-hand side of \eqref{mainidea}.
Moreover, their method is general enough to include the case of any number of $\mathrm{GL}(m)$ $L$-functions.

We will work in the opposite direction and start by noticing that relation \eqref{motibvation1} with $H=3$ and any $L_h\in\mathcal{S}_0$, $h\leq H$, follows unconditionally and for free from the results already obtained in \cite{Heap2024}. 
Indeed, based on the preceding discussion, it suffices to show that \eqref{mainidea} (and subsequently \eqref{mainidea0})  hold true for some $t\in[T,2T]$ but with $q=1$ this time.
The left-hand side is then bounded from  below as in \eqref{Hoelderlimit} while the right-hand side can be seen in view of H\"older's inequality to be bounded from above by
\[
\parentheses*{\int_{T}^{2T}\abs{L_h\parentheses{{\textstyle{\frac{1}{2}+it}}}}^{2}\abs{R\parentheses{{\textstyle{\frac{1}{2}+it}}}}^{2}\mathrm{d}t}^{1/2}\parentheses*{\int_{T}^{2T}\abs{L_g\parentheses{{\textstyle{\frac{1}{2}+it}}}}^{2}\abs{R\parentheses{{\textstyle{\frac{1}{2}+it}}}}^{2}\mathrm{d}t}^{1/2}
\]
for both of which we have \cite[Proposition 3]{Heap2024} good upper bounds.
We can argue in a similar manner for $H=4$ when at most one of the $L$-functions is associated with a cusp form  of $\mathrm{GL}(2)$  over $\mathbb{Q}$ and for $H=5$ when every $L_h$ is a Dirichlet $L$-function. 
This is the limit of the unconditional results as enforced by the sharp second moments of $\mathrm{GL}(1)$ and $\mathrm{GL}(2)$ $L$-functions and \eqref{Hoelderlimit} which solely holds for $q\geq1$.

It should be clear by now that the number of factors of the integrand in the right-hand side of  \eqref{mainidea} can become irrelevant if we know the twisted second moment of each individual factor (or any moment to that matter) and $q>0$ is sufficiently small, because an appeal to H\"older's inequality implies an upper bound which is good enough for the applications in mind.
Hence, it suffices to obtain reasonable lower bounds for {\it twisted fractional moments} of $L$-functions as depicted in the left-hand side of \eqref{mainidea} when $q\in\mathbb{Q}_{>0}$ is sufficiently small.
Sharp lower bounds for fractional moments of $\mathrm{GL}(m)$ $L$-functions have been obtained by several authors \cite{Akbary2012,Fomenko2011,Laurincikas2007,Pi2011,Sun2008} who relied on Heath-Brown's method \cite{HeathBrown1981} for the fractional moments of $\zeta(s)$ (see also \cite{Heap2022} for a novel approach).
We too shall build on this method to prove Proposition \ref{keyprop} which provides the necessary lower bounds for the twisted moments.
	\subsection{Off the critical line}
	The underlying reason for assuming the GRH when working on the critical line is so that we can express the $L$-functions (and their products) by short Dirichlet polynomials while controlling  the influence of their zeros.
	Similar operations can be performed to the right of the critical line by only assuming zero-density estimates which ensure that, as $T\to\infty$, there are horizontal strips that contain no such zeros.
	In fact these strips are wide enough that will allow us to apply the resonance method.
	Let $L_h$, $h\leq H$, be as in Theorem \ref{thm:ourresult1} and $I\subseteq[T,2T]$ be an interval such that $L_h(\sigma+it)\neq0$ for all $\sigma\geq\sigma_0$ and $t\in I$. 
		We will later see that the length of $I$ can be as large as a power of $T$.
	Arguing as in the previous subsection, in order to derive that the $L_h$, $h\leq H$, can attain simultaneously large values in the segment $[\sigma+iT,\sigma+2iT]$ it suffices to prove inequality \eqref{mainidea0} with $\sigma$ in place of ${1}/{2}$, $q=2$ and $V=\exp((D+o(1))(\log T)^{1-\sigma}/\log\log T)$.
	Dividing both sides of this inequality with $|L(\sigma+it)|^2$, we see that this is equivalent to showing that
	\[
	1\geq V^2\sum_{h\leq H}\abs{L_h\parentheses{{\textstyle{\sigma+it}}}}^{-2},\quad \text{for some }t\in I.
	\]
	A sufficient condition for the above inequality to be true can be seen to be
	\[
	\int_{I}\abs{L_h\parentheses{{\textstyle{\sigma+it}}}}^{-2}|R(it)|^2\mathrm{d}t\ll V^{-2}\int_{I}|R(it)|^2\mathrm{d}t,\quad h\leq H,
	\]
	where $R(s)$ is a Dirichlet polynomial suitably constructed to resonate with $L_h$, $h\leq H$.
	Thus, we only need to establish the latter relation while for including also the case of small values we just need to substitute $L_h^{-2}$ with $L_h^2$.

	The proof of \eqref{motivation0} and,  by extension, of Theorem \ref{thm:ourresult2} differs from those of the previous theorems.
The first step is to use a {\it convolution formula} that involves the Fej{\'e}r kernel to show roughly that, for any $h\leq H$,
		\begin{align}\label{convolcommon}
		\begin{split}
			\max_{|y|\leq\tau}\mathrm{Re}\parentheses{e^{-i\phi_h}\log L_h(s+iy)}\geq\sum_{p\asymp x}\frac{|a_{L_h}(p)|\cos(\arg(a_{L_h}(p))-t\log p)}{p^{\sigma}}+(\text{error}),
		\end{split}
	\end{align}
	where $t=\mathrm{Im}s$ ranges in an interval $I$ as described in the previous paragraph, $x\asymp\log T$ and $\tau\asymp(\log T)^{(1-\sigma)/2}\sqrt{\log\log T}$.
	The maximum in the left-hand side is the reason that makes this approach fall short in producing a single $t$ on which every $L_h$, $h\leq H$, can exhibit an extremal behavior.
	The next step in the proof would be to apply the {\it quantitative Kronecker theorem} to the  frequencies $\log p$, $p\asymp x$, in the right-hand side of \eqref{convolcommon} to show that the main term is sufficiently large.
	The immediate problem here is that the number $t$ produced by Kronecker's theorem may satisfy $t\log p\approx\arg(a_{L_h}(p))\bmod 2\pi$, $p\asymp x$, for some $h\leq H$, but this is not guaranteed for the other $L_g$, $h\neq g\leq H$.
	Mahatab et al. suggest an approach, which relies on the near orthogonality  \eqref{PNT0} of the Dirichlet coefficients, that enables us to overcome this obstacle.
	However, there is a way to avoid this approach altogether by noting that in \eqref{convolcommon} we may replace $x$ for $x_h$, $h\leq H$, where all $x_h$ are still of size $\asymp\log T$ but distant enough from each other so that $\sum_{p\asymp x_h}$ and $\sum_{p\asymp x_g}$ share no common index when $g\neq h$.
		This allows us to apply Kronecker's theorem directly, rendering the independence of the $L$-functions superfluous.
	\section{\texorpdfstring{Two classes of $L$-functions}{Two classes of L-functions}}\label{classesofL}
	The class $\mathcal{S}$, as introduced by Selberg  \cite{Selbergclass} without taking into account the Ramanujan Hypothesis, consists of ordinary Dirichlet series
	\begin{align*}
		{L}(s):=\sum\limits_{n\geq1}{a_L(n)}{n^{-s}}
	\end{align*}
	that are absolutely convergent for $\sigma>1$ and satisfy the following conditions:
	\begin{enumerate}[label=(\Roman*)]
		\item{\it Analytic continuation.} \label{finiteorder}There exists an integer $m\geq0$ such that $(s-1)^m{L}(s)$ is an entire function of finite order.
		\item{\it Functional equation.}\label{functional equation} ${L}(s)$ satisfies a relation  of the type 
		\begin{align*}
			L(s)\chi(s)=\omega\overline{\gamma(1-\overline{s})L(1-\overline{s})},\quad\gamma(s):=Q^s\prod\limits_{k\leq K}\Gamma\left(\lambda_ks+\mu_k\right),
		\end{align*}
		for some real numbers $Q, \lambda_k>0$, and complex numbers $\mu_k$, $\omega$ with $|\omega|=1$.
		\item{\it Euler product.}\label{Eulerproduct} ${L}(s)$ has an absolutely convergent product representation
		\begin{align*}
			{L}(s)=\prod_p\exp\parentheses*{\sum_{\nu\geq1}{b(p^\nu)}{p^{-\nu s}}}=:\prod_{p}L(s;p),\quad\sigma>1,
		\end{align*}
		for some complex numbers $b(p^\nu)\ll p^{\nu\theta_L }$ with $\theta_L<1/2$.
	\end{enumerate}
	The subclass $\mathcal{S}_{\text{poly}}\subseteq\mathcal{S}$ will consist of those elements from $\mathcal{S}$ which satisfy
	\begin{enumerate}[label=(\Roman*)]
		\setcounter{enumi}{3} 
		\item{\it Polynomial Euler product.}\label{pEulerproduct} There is an integer $\partial_L\geq0$ such that
		\begin{align*}
			L(s;p)=\prod_{j\leq \partial_L}\parentheses*{1-{a_{L,j}(p)}{p^{-s}}}^{-1}
		\end{align*}
		for any prime $p$  and complex numbers $a_{L,j}(p)\ll p^{\theta_L }$ with $\theta_L<1/4$.
	\end{enumerate}	
	
	In view of condition \ref{functional equation} we can define the number $d_{L}:=2\sum_{ k\leq K}\lambda_k$, which is called the {\it degree} of $L$ and  it is uniquely determined for any ${L}$-function satisfying \ref{finiteorder} and \ref{functional equation}.
	Indeed, if $N_L(T)$ counts the number of zeros of $L(s)$ in the rectangle $0\leq\sigma\leq1$, $T\leq t\leq 2T$, then 
	$
	N_L(T)=\frac{d_L}{2\pi}T\log {T}+O(T),
	$
	which is in analogy to the Riemann-von Mangoldt formula for $\zeta(s)$ \cite[\S 9.3]{Titchmarsh1986}.
	Hence, $d_L$ is well defined and it is conjectured \cite{Kaczorowski2008} that $d_L=\partial_L$ when $L\in\mathcal{S}_{\text{poly}}$. 
	Another consequence of \ref{functional equation} in combination with \ref{finiteorder} and the Phragm\'en-Lindel\"of principle is (see \cite[Theorem 6.8]{Steuding2007}) that
	\begin{align}\label{ordergrowth}
	L(\sigma+it)\ll (1+|t|)^{d_L(1-\sigma)/2+\epsilon},\quad0\leq\sigma\leq1,\,t\in\mathbb{R}.
	\end{align}
	
	Conditions \ref{Eulerproduct} and \ref{pEulerproduct} provide information on the arithmetic character of $a_L(n)$ and the analytic character of $L(s)$.
	Firstly, they imply  that $a_L(n)$, $n\geq1$, is a multiplicative function with $a_L(p)=b_L(p)\ll\sqrt{p}$ and, if $L\in\mathcal{S}_{\text{poly}}$, with $a_L(p)=\sum_{j\leq\partial_L}a_{L,j}(p)\ll p^{1/4}$ for any prime $p$. 
	In the sequel however it will also be necessary to have better control on the size of these Dirichlet coefficients.
	This will often, but not always, be realized by assuming the {\it Ramanujan Hypothesis} which asserts that $a_{L}(n)\ll n^\epsilon$ for every integer $n\geq1$.
	If in addition $L\in\mathcal{S}_{\text{poly}}$, then $|a_{L,j}(p)|\leq 1$ for all primes $p$ and integers  $j\leq\partial_L$ (see \cite[Lemma 2.2]{Steuding2007}) whence $|a_{L}(p)|\leq \partial_L$ for any prime $p$.
	The absolute convergence of the Euler product in the half-plane $\sigma>1$ yields that  $L(s)\neq0$ therein.
	Hence, by a classic result due to Landau (see for example \cite[Theorem 1]{Bombieri2011}) the function $1/L(s)$ is also represented as an ordinary Dirichlet series and as an Euler product
	\[
	\frac{1}{L(s)}:=\sum_{n\geq1}{a_{1/L}(n)}{n^{-s}}=\prod_{p}\frac{1}{L(s;p)},
	\] which are both absolutely convergent in the half-plane $\sigma>1$.
	In particular, $a_{1/L}(n)$, $n\geq1$, is also a multiplicative function satisfying $a_{1/L}(p)=-a_{L}(p)$ for any prime $p$.
	
	By Dirichlet's convolution formula it follows that, for any $L_1, L_2\in\mathcal{S}$, the functions
	$L_1L_2$ and $L_1/L_2$ can also be represented as ordinary and absolutely convergent Dirichlet series
	\begin{align*}
		L_1(s)L_2(s):=\sum_{n\geq1}{a_{L_1L_2}(n)}{n^{-s}}\quad\text{ and }\quad\frac{L_1}{L_2}(s):=\sum_{n\geq1}{a_{L_1/L_2}(n)}{n^{-s}},\quad\sigma>1,
	\end{align*}
	with the arithmetical functions 
	\[
	a_{L_1L_2}(n)=\sum_{d\mid n}a_{L_1}(n)a_{L_2}\parentheses*{\frac{n}{d}}\quad\text{ and }\quad a_{L_1/L_2}(n)=\sum_{d\mid n}a_{L_1}(n)a_{1/L_2}\parentheses*{\frac{n}{d}},\quad n\geq1,
	\]
	being multiplicative and, thus, satisfying for any prime $p$ that
	\begin{align}\label{convolution}
		a_{L_1L_2}(p)=a_{L_1}(p)+a_{L_2}(p)\quad\text{ and }\quad a_{L_1/L_2}(p)=a_{L_1}(p)-a_{L_2}(p).
	\end{align}
	Both $\mathcal{S}_{\text{poly}}\subseteq\mathcal{S}$ are multiplicitavely closed in the sense that the product of any two elements from a given class is also an element of the same class.
	Moreover, $L\in \mathcal{S}$ will be called {\it primitive} if any expression $L=L_1L_2$ for some $L_1,L_2\in\mathcal{S}$ implies that $L_1=1$ or $L_2=1$.
	
	Typical examples of primitive elements from $\mathcal{S}$, and the ones we treat in actuality here, are $L$-functions associated with irreducible cuspidal automorphic representations of $\mathrm{GL}(m)$ over $\mathbb{Q}$ (upon normalization) with $m=d_L=\partial_L$. 
	 In \cite[Section 2]{Hagen2025}, \cite[Section 2]{Heap2024} and \cite[Section 2]{Rudnick1996} can be found concise descriptions of their basic properties, while  the interested reader is referred to \cite[Chapter 5]{Iwaniec2004} for a more detailed introduction of them.
	 In a nutshell, we know that $\mathcal{S}_0\subseteq\mathcal{S}_{\text{poly}}$ with the Riemann zeta-function, Dirichlet $L$-functions and $L$-functions associated with holomorphic cusp forms satisfying the Ramanujan Hypothesis, while for $L$-functions associated with Maa{\ss}  forms it is known \cite[Appendix 2]{Kim2002} that $\theta_L\leq7/64<1/4$.
	  If $m\geq3$, then $\mathrm{GL}(m)$ $L$-functions fall marginally short outside $\mathcal{S}_{\text{poly}}$ since we know \cite{Rudnick1996} that they satisfy all four conditions \ref{finiteorder}--\ref{pEulerproduct} but with $\theta_L\leq1/2-1/(m^2+1)$.
	  Assuming the Ramanujan Hypothesis here places them automatically	in $\mathcal{S}_\text{poly}$ since then $\theta_L=0$.

	Given $L\in\mathcal{S}$  we will always be considering from here on the single-valued analytic branch of $\log L(s)$ defined on the half-plane $\sigma>1$ such that $\lim_{\sigma\to+\infty}\log L(\sigma)=0$ (by assumptions the constant Dirichlet coefficient of $L(s)$ is $1$).
	We next set 
	\[
	\Omega_L:=\set*{s:|t|\geq C_L}\setminus\bigcup_{\rho_{L}}(-\infty+i\mathrm{Im}\rho_{L},\rho_{L}],
	\]
	where $C_L>0$ is sufficiently large so that the {\it trivial zeros} of $L(s)$, i.e. its zeros that nullify the poles of the gamma factors in \ref{functional equation}, are excluded.
	Then we can extend analytically $\log L(s)$ inside $\Omega_L$ to be the value obtained from $\log L(2)$ by continuous variation along the line segments $[2,2+it]$ and $[2+it,s]$.
	\section{Extreme values on the critical line: Proof of Theorem \ref{thm:ourresult0}}\label{firsttheorem}
In this section and the succeeding one we consider a general, at first, collection of $L$-functions $L_h$, $h\leq H$, from $\mathcal{S}_{\text{poly}}$ which satisfy a stronger version of \eqref{PNT0}
		\begin{align}\label{PNT}
		\sum_{p\leq x}{a_{L_g}(p)\overline{a_{L_h}(p)}}=\left\{
		\begin{array}{ll}
			\kappa_{h}\frac{x}{\log x}+O\parentheses*{\frac{x}{(\log x)^2}},&g=h,\\
			O\parentheses*{\frac{x}{(\log x)^2}},&g\neq h,
		\end{array}
		\right.
	\end{align}
 for suitable $\kappa_{h}>0$.
	We then set
	\[
	L(s):=\prod_{h\leq H}L_h(s)
	\]
and we readily see 	from \eqref{convolution} and \eqref{PNT}  that
		\begin{align}\label{PNTP}
	\sum_{p\leq x}{|a_{L}(p)|^2}=\kappa_L\frac{x}{\log x}+O\parentheses*{\frac{x}{(\log x)^2}},\quad \kappa_L:=\sum_{h\leq H}\kappa_h.
	\end{align}
	Moreover, we require from the Dirichlet coefficients to grow by means of the following bound
	\begin{align}\label{fourthmoment}
\sum_{h\leq H} \sum_{p\leq x}{|a_{L_h}(p)|^4}{p}^{-1}\ll\log\log x.
	\end{align}
	From the discussion in the previous section we observe that the Ramanujan Hypothesis in $\mathcal{S}_{\text{poly}}$ and Merten's theorem implies \eqref{fourthmoment}.
	In general, we could work with different assumptions on the average growth of the Dirichlet coefficients, like  Hypothesis H of Rudnick and Sarnak \cite{Rudnick1996}, and end up with the same results by arguing in a similar manner.
It has been favored here to follow the approach from \cite{Heap2024} where the fourth moment bound is employed since it already covers the case of all elements of $\mathcal{S}_0$, while any other conditional result on the critical line would require first and foremost unproved assumptions on the distribution of nontrivial zeros.
		\begin{lem}\label{SpecialPNT}
			Relation \eqref{PNT} holds true for any pair $L_1, L_2\in\mathcal{S}_0$ with $\kappa_{L_1}=\kappa_{L_2}=1$.
		\end{lem}
		\begin{proof}
			Observe that
		\begin{align}\label{initialPNT}
			\sum_{p\leq x}{a_{L_1}(p)\overline{a_{L_2}(p)}}=\sum_{n\leq x}\frac{\Lambda(n)a_{L_1}(n)\overline{a_{L_2}(n)}}{\log n}-\sum_{\substack{p^\ell\leq x\\\ell\geq2}}\frac{a_{L_1}(p^\ell)\overline{a_{L_2}(p^\ell)}}{\ell},
		\end{align}
		where $\Lambda(n)$, $n\geq1$, is the von Mangoldt function.
		By definition (see \ref{pEulerproduct})
		\[
		a_{L_i}(p^\ell)=\sum_{\substack{k_1,\dots,k_{\partial_{L_i}}\geq0\\k_1+\dots+k_{\partial_{L_i}}=\ell}}a_{L_i,1}(p)^{k_1}\dots a_{L_i,\partial_{L_i}}(p)^{k_{\partial_{L_i}}}\ll p^{\ell/4}
		\]
		for any prime $p$, integer $\ell\geq1$ and $i=1,2$.
		Therefore, the second sum in the right-hand side of \eqref{initialPNT} is bounded by
		\[
		\sum_{\substack{p^\ell\leq x\\\ell\leq\log x}}\frac{p^{\ell/2}}{\ell}\ll\sqrt{x}\log\log x.
		\]
		On the other hand, the first sum in the right-hand side of \eqref{initialPNT} produces the asymptotics in \eqref{PNT} by partial summation and relation
	\begin{align*}
\sum_{n\leq x}{\Lambda(n)a_{L_1}(n)\overline{a_{L_2}(n)}}=\left\{
	\begin{array}{ll}
	x+O\parentheses*{x\exp(-c\sqrt{\log x}},&L_1=L_2,\\
		O\parentheses*{x\exp(-c\sqrt{\log x}},&L_1\neq L_2,
	\end{array}
	\right.
\end{align*}
which has been shown in \cite[Theorem 2.3]{Liu2007} for some $c>0$.
		\end{proof}
	We move on into constructing the Dirichlet polynomial that is usually utilized to pick out large values of $L(s)$. 
	We refer to \cite[Lemma 2.1]{Aistleitner2017} and in particular to \cite[p. 6366]{Heap2024} from where the aforementioned polynomial is borrowed.
	 We set $X:=T^\Delta$ for some fixed $0<\Delta<1$ and $\mathscr{L}:=\sqrt{\frac{1}{\kappa_L\lambda}\log X\log\log X}$ for some fixed $\lambda>1$.
We then consider the multiplicative function $r(n)$ to be supported only in square-free numbers and be defined by
	\[
	r(p):={a_L(p)}\frac{\mathscr{L}}{\sqrt{p}\log p},\quad p\in\mathcal{P}:=\set*{\mathscr{L}^2\leq p\leq\exp((\log \mathscr{L})^2):\max_{h\leq H}|a_{L_h}(p)|\ll(\log p)^{1-\epsilon}},
	\]
	and by $r(p)=0$ elsewhere.
	The restriction to primes where $a_{L_h}(p)$ grow at most logarithmically was imposed here only because we have not included the Ramanujan Hypothesis as a condition that the elements of $\mathcal{S}_{\text{poly}}$ should meet.
	The choice of $\mathcal{P}$ in particular ensures that $r(p)=o(1)=\max_{h\leq H}|a_{L_h}(p)r(p)|p^{-1/2}$ uniformly for $p\in\mathcal{P}$ as $T\to+\infty$ which is essential for later computations (see Lemma \ref{resonator properties}).
	Other than that, the resonating polynomial is defined as usual to be
	\begin{align}\label{resonator0}
		R(s)=\sum_{n\leq X}{r(n)}{n^{1/2-s}}
	\end{align}
	and we can show, under the assumptions \eqref{PNT} and \eqref{fourthmoment}, the following propositions.
		\begin{prop}\label{essprop}
	If $L_h\in\mathcal{S}_0$, for some $h\leq H$, then for any sufficiently small but fixed $\Delta>0$, any fixed $\lambda>1$ and for any sufficiently large $T\gg1$ we have that
		\[
		\frac{1}{T}\int_\mathbb{R}\abs*{L_h\parentheses*{{\textstyle{\frac{1}{2}+it}}}}^{2}\abs*{R\parentheses*{{\textstyle{\frac{1}{2}+it}}}}^{2}\mathrm{d}t\leq  \exp\parentheses*{\parentheses{2+o(1)}\sqrt{\frac{\Delta\log T}{\kappa_L\lambda\log\log T}}}\prod_{p}\parentheses*{1+|r(p)|^2}.
		\]
		If $L_h$ is a Dirichlet $L$-function, then any $\Delta<{17}/{33}$ is admissible. 
		Otherwise, we can take $\Delta<\frac{ 1/2-\theta}{3+\theta}$ where $\theta$ is an admissible bound towards the Ramanujan Hypothesis for cusp forms on $G L(2)$ over $\mathbb{Q}$.
	\end{prop}
	\begin{proof}
		This is  \cite[Proposition 3]{Heap2024} where it is shown that
		\begin{align*}
		\frac{1}{T}\int_T^{2T}\abs{L_h\parentheses{{\textstyle{\frac{1}{2}+it}}}}^{2}\abs{R\parentheses{{\textstyle{\frac{1}{2}+it}}}}^{2}\mathrm{d}t\ll \prod_{p}\parentheses*{1+|r(p)|^2+(2+o(1))\frac{\mathrm{Re}\parentheses{\overline{r(p)}a_{L_h}(p)}}{\sqrt{p}}}.
		\end{align*}
		It makes no difference that we do not require for the rest of the $L_g$, $g\neq h$, to be $\mathrm{GL}(m)$ $L$-functions because we assume the same, sufficient for the proof, conditions for the elements of $\mathcal{S}_{\text{poly}}$.
	The right-hand side of the above relation is equal to
		\[
		\brackets*{\prod_{p}\parentheses*{1+|r(p)|^2}}\prod_{p}\parentheses*{1+(2+o(1))\frac{\mathrm{Re}\parentheses{\overline{r(p)}a_{L_h}(p)}}{\sqrt{p}\parentheses*{1+|r(p)|^2}}},
		\]
		where the second product is bounded from above by the desired factor in the statement of the proposition as can be seen from \ref{resonator properties(i)}, \ref{resonator properties(iii)} and Lemma \ref{SpecialPNT}.
	\end{proof}
	\begin{prop}\label{keyprop}
		Let $\pi\subseteq\set{1,\dots,H}$ and $q=u/v\in\mathbb{Q}_{>0}$.
	If $L_\pi(s):=\prod_{h\in\pi}L_h(s)$ with $\kappa_\pi:=\sum_{h\in\pi}\kappa_{h}$,
	then for fixed $0<\Delta<1/2$ and $\lambda>v$, and for any sufficiently large $T\gg1$
	\[
	\frac{1}{T}\int_\mathbb{R}\abs{L_\pi\parentheses{{\textstyle{\frac{1}{2}+it}}}}^{q}\abs{R\parentheses{{\textstyle{\frac{1}{2}+it}}}}^{2}\mathrm{d}t\geq \exp\parentheses*{\parentheses{q\kappa_\pi+o(1)}\sqrt{\frac{\Delta\log T}{\kappa_L\lambda\log\log T}}}\prod_{p}\parentheses*{1+|r(p)|^2}.
	\]
	\end{prop}
	We postpone its proof  in the next section and conclude  with the proof of Theorem \ref{thm:ourresult0}.
	\begin{proof}[Proof of Theorem \ref{thm:ourresult0}]
		By Lemma \ref{SpecialPNT} we know that $L_h$, $h\leq H$, satisfy \eqref{PNT} with $\kappa_h=1$.
		The same holds true with respect to \eqref{fourthmoment} (see \cite[(2.10)]{Heap2024}).
		Set
		\[
			Z:=\sqrt{\frac{D\log T}{H(H-1)\log\log T}}\quad\text{ and }\quad\tilde{Z}:=\sqrt{\frac{\Delta\log T}{H\lambda\log\log T}},
		\]
			where $D>0$ is any fixed number in the admissible range while $\Delta$ and $\lambda$ will soon be fixed.
			The theorem will follow  if we can show that for any $T\gg1$, there is $t\in[T,2T]$ with
		\begin{align*}
			\abs{L\parentheses{{\textstyle{\frac{1}{2}+it}}}}^{2/(H-1)}> \exp\parentheses*{\parentheses*{\frac{2}{H-1}+o(1)}Z}\sum_{h\leq H}\prod_{g\neq h}\abs{L_g\parentheses{{\textstyle{\frac{1}{2}+it}}}}^{2/(H-1)}.
		\end{align*}
		A sufficient condition for the above inequality to hold is to prove that, for any $h\leq H$,
		\begin{align*}
			\begin{split}
				\mathscr{I}_h:=\frac{\int_T^{2T}	\abs{L\parentheses{{\textstyle{\frac{1}{2}+it}}}}^{2/(H-1)}\abs{R\parentheses{{\textstyle{\frac{1}{2}+it}}}}^{2}\mathrm{d}t}{\int_{T}^{2T}\abs{R\parentheses{{\textstyle{\frac{1}{2}+it}}}}^{2}\prod_{g\neq h}\abs{L_g\parentheses{{\textstyle{\frac{1}{2}+it}}}}^{2/(H-1)}\mathrm{d}t}
				>\exp\parentheses*{\parentheses*{\frac{2}{H-1}+o(1)}Z}.
			\end{split}
		\end{align*}
		Towards that goal we fix $\Delta\to D^+$ in the admissible range and $\lambda\to(H-1)^{+}$ so that $\frac{\Delta}{\lambda}>\frac{D}{H-1}$.
		Proposition  \ref{keyprop} yields then that
		\[
		\int_T^{2T}\abs{L\parentheses{{\textstyle{\frac{1}{2}+it}}}}^{2/(H-1)}\abs{R\parentheses{{\textstyle{\frac{1}{2}+it}}}}^{2}\mathrm{d}t\geq T \exp\parentheses*{\parentheses*{\frac{2H}{H-1}+o(1)}\tilde{Z}}\prod_{p}\parentheses*{1+|r(p)|^2},
		\]
		since $\kappa_L=H$.
		On the other hand, H\"older's inequality and Proposition \ref{essprop} imply that
		\begin{align*}
\int_{T}^{2T}\abs{R\parentheses{{\textstyle{\frac{1}{2}+it}}}}^{2}\prod_{g\neq h}\abs{L_g\parentheses{{\textstyle{\frac{1}{2}+it}}}}^{2/(H-1)}\mathrm{d}t
		&\leq\prod_{g\neq h}\parentheses*{\int_{T}^{2T}\abs{L_g\parentheses{{\textstyle{\frac{1}{2}+it}}}}^{2}\abs{R\parentheses{{\textstyle{\frac{1}{2}+it}}}}^{2}\mathrm{d}t}^{1/(H-1)}\\
		&\leq T\exp\parentheses*{\parentheses*{{2}+o(1)}\tilde{Z}}\prod_{p}\parentheses*{1+|r(p)|^2}.
		\end{align*}
		Therefore,  
		\[
		\mathscr{I}_h\geq \exp\parentheses*{\parentheses*{\frac{2}{H-1}+o(1)}\tilde{Z}}> \exp\parentheses*{\parentheses*{\frac{2}{H-1}+o(1)}{Z}}.
		\]
	\end{proof}
	
	\section{Proof of Proposition \ref{keyprop}}\label{seckeyprop}
	We  start by setting up the necessary background.
	We have that
	\[
	\log L_{\pi}(s)=\sum_{h\in\pi}\sum_{p}\sum_{j\leq\partial_{L_h}}\sum_{k\geq1}{a_{L_h,j}(p)^k}{p^{-ks}}=\sum_{h\in\pi}\sum_{p}\sum_{j\leq\partial_{L_h}}\log\parentheses*{1-{a_{L_h,j}(p)}{p^{-s}}}^{-1},\quad\sigma>1,
	\]
	where the branch of $\log (1-z)$, $|z|<1$, is the one that vanishes at $z=0$ and which is defined in any other point $w$ of the unit disk by continuous variation along the segment $[0,w]$.
	Hence, for $q>0$ we can also define a branch of $L_{\pi}(s)^q$ by
	\begin{align*}
		L_\pi(s)^{q}&=\exp(q \log L_\pi(s))\\
		&=\prod_{p}\prod_{h\in\pi}\prod_{j\leq\partial_{L_h}}\parentheses*{1-{a_{L_h,j}(p)}{p^{-s}}}^{-q}\\
		&=\prod_{p\in\mathcal{P}}\prod_{p\notin\mathcal{P}}\prod_{h\in\pi}\prod_{j\leq\partial_{L_h}}\sum_{k\geq0}\frac{\Gamma(q+k)}{k!\Gamma(q)}{a_{L_h,j}(p)^k}{p^{-ks}}\\
		&=:\sum_{n\geq1}{c_{1}(n)}{n^{-s}}\sum_{n\geq1}{c_{2}(n)}{n^{-s}}.
	\end{align*}
	Here $c_{1}(n)$ and $c_2(n)$, $n\geq1$, are two multiplicative functions which are respectively supported on positive integers divisible only by primes $p\in\mathcal{P}$ and only by primes $p\notin\mathcal{P}$. 
	For any prime $p\in\mathcal{P}$ and integer $\nu\geq1$  we have that $c_2(p^\nu)=0$ and
	\begin{align}\label{defcoef}
		c_{1}(p^\nu)=\sum\cdots\sum\nolimits^\nu\prod_{h\in\pi}\prod_{j_h\leq\partial_{L_h}}\tau_q(p^{k_{j_h}})a_{L_h,j_h}(p)^{k_{j_h}},\quad \tau_q(p^k):=\frac{\Gamma(q+k)}{k!\Gamma(q)},
	\end{align}
	where the multisum $\sum\cdots\sum^\nu$ is over all tuples $((k_{j_h})_{j_h\leq\partial_{L_h}})_{h\in\pi}$ of non-negative integers such that $\nu=\sum_{h\in\pi}\sum_{j_h\leq\partial_{L_h}}k_{j_h}$.
	The function $\tau_q(n)$ is the generalization of the divisor function $\tau_2(n)$ and it satisfies $0<\tau_q(n)\ll_q n^\epsilon$ for any integer $n\geq1$ and real $q>0$.
	We then have the following relations
	\begin{align}\label{fractcoeffic}
		c_1(p)=q\sum_{h\in\pi}a_{L_h}(p)\quad\text{ and }\quad c_1(p^\nu)\ll p^{\nu\theta_\pi},
	\end{align}
	where $\theta_\pi$ is the center of the interval $\parentheses*{\max_{h\in\pi}\theta_{L_h},{1}/{4}}.$
	In fact, the construction of $\mathcal{P}$ yields near logarithmic bounds for $c_1(p^\nu)$ but it makes no difference in the succeeding arguments.
	If $p\notin\mathcal{P}$ and $\nu\geq1$, then $c_1(p^\nu)=0$ and \eqref{defcoef}--\eqref{fractcoeffic} hold for $c_2(p^\nu)$ instead.
	
We introduce next the positive weight function
	\[
	w(t,T)=\int_{6T/5}^{9T/5}e^{-(t-\tau)^2}\mathrm{d}\tau,\quad T>0,\,t\in\mathbb{R},
	\]
	which satisfies the following properties
	\begin{itemize}
	\item $w(t,T)\ll 1$, for any $T>0$ and $t\in\mathbb{R}$,
		\item $w(t,T)\gg1$, for any $T>0$ and $6T/5\leq t\leq 9T/5$,
		\item $w(t,T)\ll e^{-\epsilon(t^2+T^2)}$, for any $T>0$ and $t\leq 7T/6$ or $t\geq11T/6$.
	\end{itemize}
	If we set 
	\[
	I_{\pi}(\sigma,T):=\frac{1}{T}\int_\mathbb{R}\abs*{L_\pi(\sigma+it)}^{q}\abs*{R(\sigma+it)}^2w(t,T)\mathrm{d}t,
	\]
	then we have in view of \eqref{ordergrowth} and \eqref{resonator0} that
	\begin{align}\label{intermediate}
		\begin{split}
			I_\pi({\textstyle{\frac{1}{2}}},T)
			&\ll\frac{1}{T}\int_{7T/6}^{11T/6}\abs{L_\pi(\textstyle{\frac{1}{2}}+it)}^{q}\abs{R(\textstyle{\frac{1}{2}}+it)}^2\mathrm{d}t\\
			&\quad+\frac{1}{T}\int_{t\notin[7T/6,11T/6]}(1+|t|)^{q\sum_{h\in\pi}d_{L_h}}T^{2\Delta}e^{-\epsilon(t^2+T^2)}\mathrm{d}t\\
			&\ll\frac{1}{T}\int_{T}^{2T}\abs{L_\pi(\textstyle{\frac{1}{2}}+it)}^{q}\abs{R(\textstyle{\frac{1}{2}}+it)}^2\mathrm{d}t+e^{-\epsilon T^2}.
		\end{split}
	\end{align}
	Hence, it suffices to prove Proposition \ref{keyprop} for $I_\pi(\frac{1}{2},T)$ instead.
	To that end, let
	\[
	P_i(s,T):=\sum_{n\leq X}c_i(n)n^{-s},\quad	s_\pi(s,T):=P_1(s,T)P_2(s,T),\quad d_\pi(s,T):=L_\pi(s)^u-s_\pi(s,T)^v,
	\]
	and
	\begin{align*}
		S_\pi(\sigma,T)&:=\frac{1}{T}\int_\mathbb{R}\abs*{s_{\pi}(\sigma+it,T)}\abs*{R(\sigma+it)}^2w(t,T)\mathrm{d}t,\\ D_\pi(\sigma,T)&:=\frac{1}{T}\int_\mathbb{R}\abs*{d_{\pi}(\sigma+it,T)}^{1/v}\abs*{R(\sigma+it)}^2w(t,T)\mathrm{d}t.
	\end{align*}
	Recall also here Gabriel's convexity theorem \cite[Theorem 2]{Gabriel1927} and the Montgomery-Vaughan mean value theorem for ordinary Dirichlet series \cite[Corollary 3]{Montgomery1974}.
	\begin{lem}\label{Gabriel0}
		Let $f(s)$ be analytic in the vertical strip $\sigma_1<\sigma<\sigma_2$, and continuous for $\sigma_1\leq\sigma\leq\sigma_2$.
		Suppose that $f(s)\to0$ as $|\mathrm{Im}(s)|\to\infty$ uniformly for $\sigma_1\leq\sigma\leq\sigma_2$. 
		Then for any $\sigma_1\leq\sigma\leq\sigma_2$ and any $y>0$ we have that
		\[
		\int_\mathbb{R}|f(\sigma+it)|^{y}\mathrm{d}t\leq\parentheses*{\int_\mathbb{R}|f(\sigma_1+it)|^{y}\mathrm{d}t}^{\frac{\sigma_2-\sigma}{\sigma_2-\sigma_1}}\parentheses*{\int_\mathbb{R}|f(\sigma_2+it)|^{y}\mathrm{d}t}^{\frac{\sigma-\sigma_1}{\sigma_2-\sigma_1}}.
		\]
	\end{lem}
	\begin{lem}\label{M-V}
		If $\sum_{n\geq1}n|a_n|^2<\infty$, then
		\[
		\int_0^T\abs*{\sum_{n\geq1}a_nn^{-it}}^2\mathrm{d}t=\sum_{n\geq1}|a_n|^2(T+O(n)).
		\]
	\end{lem}
	We can now prove the following preparatory lemmas.
	\begin{lem}\label{Gabriel}
		Let $\beta>3/2$ be fixed. 
		We have, uniformly for $1/2\leq\sigma\leq\beta$ and $T\gg1$, that
		\[
		I_\pi(\sigma, T)\ll I_\pi\parentheses{{\textstyle\frac{1}{2}},T}^{1+\frac{1/2-\sigma}{\beta-1/2}}+e^{-\epsilon T^2}.
		\]
	\end{lem}
	\begin{proof}
		The proof is the same as of \cite[Lemma 5]{HeathBrown1981} with $L(s)^uR(s)^{2v}$ in place of $\zeta(s)$.
		We briefly repeat it here for the sake of completeness and because it will also serve as reference for the proof of the next lemma.
		If we set $f(s):=(s-1)^{mu}L_\pi(s)^uR(s)^{2v}\exp\parentheses*{v(s-i\tau)^2}$, where $m$ is the degree of the pole of $L_\pi(s)$ at $s=1$ and $\tau\geq2$, then  \eqref{ordergrowth} and \eqref{resonator0} imply that $f(s)$ satisfies the assumptions of Lemma \ref{Gabriel0} in the strip $1/2\leq\sigma\leq\beta$.
		Hence,
		\[
		\int_\mathbb{R}|f(\sigma+it)|^{1/v}\mathrm{d}t\leq\parentheses*{\int_\mathbb{R}\abs{f\parentheses{{\textstyle{\frac{1}{2}}}+it}}^{1/v}\mathrm{d}t}^{\frac{\beta-\sigma}{\beta-1/2}}\parentheses*{\int_\mathbb{R}|f(\beta+it)|^{1/v}\mathrm{d}t}^{\frac{\sigma-1/2}{\beta-1/2}}.
		\]
		Observe that for $\xi\in\set{1/2,\beta}$ we have that
		\begin{align*}
			\int_\mathbb{R}\abs*{f\parentheses{\xi+it}}^{1/v}\mathrm{d}t
			&\ll\int_\mathbb{R}\parentheses{1+|t|}^{mq}\abs*{L_\pi\parentheses*{\xi+it}}^{q}\abs*{R\parentheses*{\xi+it}}^2e^{-(t-\tau)^2}\mathrm{d}t\\
			&\ll\tau^{mq}\int_{5\tau/6}^{10\tau/9}\abs*{L_\pi\parentheses*{\xi+it}}^{q}\abs*{R\parentheses*{\xi+it}}^2e^{-(t-\tau)^2}\mathrm{d}t+T^{2\Delta}e^{-\epsilon\tau^2}
		\end{align*}
		due to the exponential decay of the integrand when $t\notin[5\tau/6,10\tau/9]$.
		For the same reason
		\begin{align*}
			J(\sigma,\tau)&:=\int_\mathbb{R}\abs*{L_\pi\parentheses*{\sigma+it}}^{q}\abs*{R\parentheses*{\sigma+it}}^2e^{-(t-\tau)^2}\mathrm{d}t\\
			&\ll	\int_{5\tau/6}^{10\tau/9}\abs*{L_\pi\parentheses*{\sigma+it}}^{q}\abs*{R\parentheses*{\sigma+it}}^2e^{-(t-\tau)^2}\mathrm{d}t+T^{2\Delta}e^{-\epsilon\tau^2}\\
			&\ll \tau^{-mq}	\int_\mathbb{R}\abs*{f\parentheses*{\sigma+it}}^{1/v}\mathrm{d}t+T^{2\Delta}e^{-\epsilon\tau^2}.
		\end{align*}
		From the above relations and inequality $(x+y)^a\leq 2^{a-1}(x^a+y^a)$, valid for any positive numbers $x,y$ and $a\leq1$, we derive that
		\begin{align*}
			J(\sigma,\tau)
			&\ll J\parentheses{{\textstyle{\frac{1}{2}}},\tau}^{\frac{\beta-\sigma}{\beta-1/2}}\parentheses*{\int_{5\tau/6}^{10\tau/9}\abs*{L_\pi\parentheses*{\beta+it}}^{q}\abs*{R\parentheses*{\beta+it}}^2e^{-(t-\tau)^2}\mathrm{d}t+\frac{T^{2\Delta}}{e^{\epsilon\tau^2}}}^{\frac{\sigma-1/2}{\beta-1/2}}+\frac{T^{2\Delta}}{e^{\epsilon\tau^2}}.
		\end{align*}
		We then integrate both sides of the latter relation over $6T/5\leq\tau\leq 9T/5$, apply H\"older's inequality to the right-hand side of the resulting relation (to insert the integration inside the powers) and lastly divide by $T$ to conclude that
		\begin{align*}
			I_\pi(\sigma,T)\ll I_\pi\parentheses{{\textstyle{\frac{1}{2}}},T}^{\frac{\beta-\sigma}{\beta-1/2}}\parentheses*{\frac{1}{T}\int_{T}^{2T}\abs*{L_\pi\parentheses*{\beta+it}}^{q}\abs*{R\parentheses*{\beta+it}}^2\mathrm{d}t+e^{-\epsilon T^2}}^{\frac{\sigma-1/2}{\beta-1/2}}+e^{-\epsilon T^2}.
		\end{align*}
		The lemma follows now by noting that  $\abs*{L_\pi\parentheses*{\beta+it}}^{q}\abs*{R\parentheses*{\beta+it}}^2\ll1$, $t\in\mathbb{R}$, $\beta>3/2$. 
	\end{proof}
	\begin{lem}\label{lowerboundapprox2}
		Let $\beta\gg1$ be any sufficiently large but fixed positive number. 
		We then have, uniformly for $1/2\leq\sigma\leq\beta$ and $T\gg1$, that
		\begin{align*}
			D_\pi(\sigma,T)
			&\ll D_\pi({\textstyle{\frac{1}{2}}},T)^{1+\frac{1/2-\sigma}{\beta-1/2}}T^{(1/2-\sigma)\parentheses*{1-1/\beta}\Delta/v}+e^{-\epsilon T^{2}}.
		\end{align*}
	\end{lem}
	\begin{proof}
		
		This time we set $f(s)=	(s-1)^{mu}d_\pi(s,T)R(s)^{2v}\exp\parentheses*{v(s-i\tau)^2}$. 
		Once more we can employ Lemma \ref{Gabriel0} and by arguing exactly as in Lemma \ref{Gabriel}, we can show that 
		\begin{align}\label{D1}
			D_\pi(\sigma,T)
			\ll D_\pi({\textstyle{\frac{1}{2}}},T)^{\frac{\beta-\sigma}{\beta-1/2}}\parentheses*{\frac{1}{T}\int_{T}^{2T}\abs*{d_\pi(\beta+it,T)}^{1/v}\abs*{R(\beta+it)}^2\mathrm{d}t+e^{-\epsilon T^2}}^{\frac{\sigma-1/2}{\beta-1/2}}+e^{-\epsilon T^2}.
		\end{align}
		The factor $\abs*{R(\beta+it)}^2$ in the integrand is $O(1)$ when $\beta>3/2$ and so we exclude it from the succeeding estimations.
		An application of Hölder's inequality yields then that the above integral is bounded by
		\begin{align}\label{D2}
			J:=\parentheses*{\frac{1}{T}\int_{T}^{2T}\abs*{L_\pi\parentheses*{\beta+it}^u-s_\pi\parentheses*{\beta+it,T}^v}^{2}\mathrm{d}t}^{1/(2v)}.
		\end{align}
		Note that, for $\sigma>1$, $L_\pi\parentheses*{s}^u-s_\pi\parentheses*{s,T}^v
		=		\parentheses*{L_\pi\parentheses*{s}^q}^v-s_\pi\parentheses*{s,T}^v$ which can be written as
		\begin{align*}
		\parentheses*{\sum_{n\geq1}c_2(n)n^{-s}}^v\brackets*{\parentheses*{\sum_{n\geq1}c_1(n)n^{-s}}^v-P_1(s)^v}+P_1(s)^v\brackets*{\parentheses*{\sum_{n\geq1}c_2(n)n^{-s}}^v-P_2(s)^v}.
		\end{align*}
Since all terms are $O(1)$, we obtain that
		\[
		J^{2v}\ll\sum_{i\leq2}\frac{1}{T}\int_{T}^{2T}\abs*{\parentheses*{\sum_{n\geq1}c_i(n)n^{-\beta-it}}^v-\parentheses*{\sum_{n\leq X}c_i(n)n^{-\beta-it}}^v}^2\mathrm{d}t.
		\]
	By construction, the integrands above can be expressed as 
		\[
		\abs*{\sum_{n>X}{C_{i}(n)}{n^{-\beta-it}}}^2,\quad
		C_i(n):=\sum_{\substack{n_1,\dots,n_v\geq1\\n_1\cdots n_v=n\\\exists j\leq v:\,n_j>X}}c_i(n_1)\cdots c_i(n_v)\ll n^{\theta_\pi},
		\]
		with the upper bound being derived from \eqref{fractcoeffic}.
	Lemma \ref{M-V} yields then that
		\begin{align}\label{D3}
		J^{2v}
			\ll\sum_{i\leq2}\sum_{n>X}\frac{\abs*{C_i(n)}^2}{n^{2\beta}}\parentheses*{1+\frac{n}{T}}
			\ll T^{-{\Delta}(2\beta-1)\parentheses*{1-\frac{2\theta_\pi}{2\beta-1}}}
		\end{align}
		and the lemma follows from \eqref{D1}--\eqref{D3}.
	\end{proof}
		Before moving to the next two lemmas, we introduce some additional notation. 
		Let $\eta:=(\log\mathscr{L})^{-3}$ and
		\[
		\mathscr{V}_\pi(\sigma):=\mathscr{R}(\sigma)\exp\parentheses*{\mathscr{F}_{\pi}(\sigma)},
		\] where
	\begin{align*}
		\mathscr{R}(\sigma):=\prod_{p}\parentheses*{1+\frac{|r(p)|^2}{p^{2\sigma-1}}}\quad\text{ and }\quad\mathscr{F}_\pi(\sigma):=\sum_{p}\frac{\mathrm{Re}\parentheses{\overline{r(p)}c_{1}(p)}}{p^{2\sigma-1/2}\parentheses*{1+\frac{|r(p)|^2}{p^{2\sigma-1}}}}.
		\end{align*}
	In view of \eqref{fractcoeffic},  \ref{resonator properties(vi)} and \ref{resonator properties(iii)} we have, uniformly for $T\gg1$ and $1/2\leq\sigma\leq1/2+\eta$, that
	\begin{align}\label{mainasymp}
		\mathscr{F}_\pi(\sigma)=\mathscr{F}_\pi\parentheses*{\textstyle{\frac{1}{2}}}+O\parentheses*{\frac{\sqrt{\log X}}{\log\log X}}=\parentheses*{q\kappa_{\pi}+o(1)}\sqrt{\frac{\log X}{\kappa_L\lambda\log\log X}}.
	\end{align}
	We proceed to show that $S_\pi(\sigma, T)$ is, up to logarithmic powers, bounded from below and above by $\mathscr{V}_\pi(\sigma)$ in a uniform range of $\sigma$ close to $1/2$. 
		The proof of the lower bound is standard (see for instance \cite[Section 4]{Heap2024}, \cite[Lemma 2.1]{Aistleitner2017}, \cite[Theorem 2.1]{Soundararajan2008}) and we will therefore sketch here briefly the arguments with suitable adjustments.
	\begin{lem}\label{lowerboundapprox1}
		We have,  uniformly for $T\gg1$ and $1/2\leq\sigma\leq1/2+\eta$, that
		\[
		S_\pi(\sigma, T)\gg\mathscr{V}_\pi(\sigma).
		\]
	\end{lem}
	
	\begin{proof}
		Let $0<U<V$ be fixed and observe that
		\begin{align}\label{mainsum0}
				T\mathscr{U}:=\int_{VT}^{UT}\abs*{s_\pi(\sigma+it,T)}\abs*{R(\sigma+it)}^2\mathrm{d}t\geq\abs*{\int_{\mathbb{R}}{s_\pi(\sigma+it,T)}\abs*{R(\sigma+it)}^2\Phi\parentheses*{\frac{t}{T}}\mathrm{d}t},
		\end{align}
		where $\Phi\in C^\infty(\mathbb{R})$ is compactly supported in $[U,V]$ with $\hat{\Phi}(0)\neq0$.
		Repeated integration by parts implies  that 
		$
		\hat{\Phi}(y):=\int_{\mathbb{R}}\Phi(x)e^{-2\pi ixy}\mathrm{d}x\ll_Q(1+|y|)^{-Q}$, $y\in\mathbb{R}$, for any integer $Q\geq1$.
		
		Opening the square in the right-hand side of \eqref{mainsum0} and interchanging summations and integration yields that
			\begin{align}\label{mainsum}
		\frac{1}{T}\int_\mathbb{R}s_{\pi}(\sigma+it,T)\abs*{R(\sigma+it)}^2\Phi\parentheses*{\frac{t}{T}}\mathrm{d}t=\hat{\Phi}(0)\mathop{\sum_{k,\ell,m,n\leq X}}_{\ell mn=k}\frac{c_1(n)c_2(\ell)r(m)\overline{r(k)}}{(n\ell)^\sigma(mk)^{\sigma-1/2}}+E,
		\end{align}
		where
		\[
		E:=\mathop{\sum_{k,\ell,m,n\leq X}}_{\ell mn\neq k}\frac{c_1(n)c_2(\ell)r(m)\overline{r(k)}}{(n\ell)^\sigma(mk)^{\sigma-1/2}}\hat{\Phi}\parentheses*{T\log\frac{\ell mn}{k}}=o(1),
		\]
		as follows from the bound $T\abs*{\log(\ell mn/k)}\gg TX^{-2}\gg T^\epsilon$, valid for any $k,\ell,m,n\leq X$ with $\ell mn\neq k$, the polynomial growth of the coefficients of $E$ and the fast decay of $\hat{\Phi}(y)$.
		
		For the diagonal contribution $\ell mn=k$ that corresponds to the multisum in the right-hand side of \eqref{mainsum} we note that $(\ell,kmn)=1$ by construction.
		Therefore, the only term of the $\ell$-sum that contributes is $c_2(1)=1$ for $\ell=1$.
		Taking also into account that $c_1(n)$ and $r(n)$ are multiplicative, the latter being  supported in square-free integers, we can rewrite the multisum as
		\begin{align}\label{mainsum1}
	\sum_{n\leq X}\frac{\overline{r(n)}c_1(n)}{n^{2\sigma-1/2}}\mathop{\sum_{m\leq \frac{X}{n}}}_{(m,n)=1}\frac{|r(m)|^2}{m^{2\sigma-1}}=M+O(E_1+E_2),
		\end{align}
		where
		\begin{align}
		M:=\sum_{n\geq1}\frac{\overline{r(n)}c_1(n)}{n^{2\sigma-1/2}}\prod_{p\nmid n}\parentheses*{1+\frac{|r(p)|^2}{p^{2\sigma-1}}}=\mathscr{R}(\sigma)\prod_{p}\parentheses*{1+\frac{\overline{r(p)}c_1(p)}{p^{2\sigma-1/2}\parentheses*{1+\frac{|r(p)|^2}{p^{2\sigma-1}}}}},
		\end{align}
		\begin{align*}
E_1:=\sum_{n\leq X}\frac{|c_1(n){r(n)|}}{n^{2\sigma-1/2}}\mathop{\sum_{m> \frac{X}{n}}}_{(m,n)=1}\frac{|r(m)|^2}{m^{2\sigma-1}}\quad\text{ and }\quad E_2:=\sum_{n>X}\frac{|c_1(n){r(n)|}}{n^{2\sigma-1/2}}\sum_{\substack{m\geq1\\(m,n)=1}}\frac{|r(m)|^2}{m^{2\sigma-1}}.
		\end{align*}
		Applying Rankin's trick\footnote{If $a_n\geq0$, $n\geq1$, then $\sum_{n>X}a_n\leq X^{-\alpha}\sum_{n\geq1}a_nn^{\alpha}$ for any positive numbers $X\geq1$ and $\alpha$.} once in the $m$-sum of $E_1$ and once in the $n$-sum of $E_2$ yields that
		\begin{align}
			\begin{split}
		E_1+E_2&\ll X^{-\eta}\prod_p\parentheses*{1+\frac{|r(p)|^2p^\eta}{p^{2\sigma-1}}+\frac{|{r(p)}c_1(p)|p^\eta}{p^{2\sigma-1/2}}}\\
		&\ll X^{-\eta}\mathscr{R}(\sigma)\prod_p\parentheses*{1+\frac{|r(p)|^2}{p^{2\sigma-1}}(p^\eta-1)}\parentheses*{1+\frac{|{r(p)}c_1(p)|p^\eta}{p^{2\sigma-1/2}\parentheses*{1+\frac{|r(p)|^2p^\eta}{p^{2\sigma-1}}}}}.
		\end{split}
		\end{align}
		Since $p^\eta=O(1)$, $p\in\mathcal{P}$, we can employ relations  \ref{resonator properties(i)}, \ref{resonator properties(v)}, \eqref{fractcoeffic}  and \ref{resonator properties(iv)} to show that the logarithm of the above $p$-product is equal to
		\begin{align}\label{mainsum2}
			\eta\parentheses*{\frac{1}{\lambda}+o(1)}\log X.
		\end{align}

		Hence, in view of \eqref{mainsum0}--\eqref{mainsum2}  and since $\lambda>1$ we deduce that
		\[
	\mathscr{U}\gg\mathscr{R}(\sigma)\prod_{p}\abs*{1+\frac{\overline{r(p)}c_1(p)}{p^{2\sigma-1/2}\parentheses*{1+\frac{|r(p)|^2}{p^{2\sigma-1}}}}}+o(1)\mathscr{R}(\sigma).
		\]
		It remains to see that the logarithm of the $p$-product in the right-hand side above equals to
		\begin{align*}
		\mathrm{Re}\parentheses*{\sum_{p}\log\parentheses*{1+\frac{\overline{r(p)}c_1(p)}{p^{2\sigma-1/2}\parentheses*{1+\frac{|r(p)|^2}{p^{2\sigma-1}}}}}}
		=\mathscr{F}_{\pi}(\sigma)+E_3,
		\end{align*}
		where
		\begin{align*}
			E_3\ll\sum_{p\in\mathcal{P}}\frac{|r(p)|^2|c_1(p)|^2}{p^{4\sigma-1}}\ll\mathscr{L}^2\sum_{h\in\pi}\sum_{p\geq\mathscr{L}^2}\frac{|a_{L_h}(p)|^2}{p^2(\log p)^{2\epsilon}}=o(1),
		\end{align*}
		as follows from \eqref{PNT}.
		
	Therefore, $\mathscr{U}\gg\mathscr{R}(\sigma)\exp(\mathscr{F}_\pi(\sigma))$ which suffices to derive the lemma because, in view of the properties of $w(t,T)$, we have that $S_\pi(\sigma,T)\gg\mathscr{U}$ for $U=6T/5$ and $V=9T/5$.
		\end{proof}
		\begin{lem}\label{upperboundapprox1}
			We have,  uniformly for $T\gg1$ and $1/2\leq\sigma\leq1/2+\eta$, that
			\[
			S_\pi(\sigma, T)\ll\mathscr{V}_\pi\parentheses*{\sigma}(\log T)^{q^2\kappa_\pi/2+o(1)}.
			\]
		\end{lem}
		\begin{proof}
			If $\mathscr{U}$ is as before, then H\"older's inequality yields that
			\[
			\mathscr{U}\leq\parentheses*{\frac{1}{T}\int_{UT}^{VT}\abs*{P_1(\sigma+it)}^2\abs*{R(\sigma+it)}^2\mathrm{d}t}^{1/2}\parentheses*{\frac{1}{T}\int_{UT}^{VT}\abs*{P_2(\sigma+it)}^2\abs*{R(\sigma+it)}^2\mathrm{d}t}^{1/2}.
			\]
			Lemma \ref{M-V} implies for both integrals that
		\[
	\mathscr{U}_i:={\frac{1}{T}\int_{UT}^{VT}\abs*{P_i(\sigma+it)R(\sigma+it)}^2\mathrm{d}t}=\parentheses*{1+O\parentheses*{\frac{X^2}{T}}}\sum_{k\leq X^2}\frac{|u_i(k)|^2}{k^{2\sigma}},\quad i\leq2,
		\]
		where
		\[
		|u_i(k)|^2:=\mathop{\sum_{m_1,m_2,n_1,n_2\leq X}}_{m_1 n_1=m_2n_2=k}c_i(m_1)r(n_1)\sqrt{n_1}\overline{c_i(m_2)r(n_2)\sqrt{n_2}}.
		\]
		Observe that when $i=2$, for the integers $m_1,m_2,n_1,n_2$ of the respective multisum we have that $(m_1m_2,n_1n_2)=1$.
		Therefore, the condition $m_1n_1=m_2n_2$ simply implies that $m_1=m_2$ and $n_1=n_2$. 
		Taking also into account the multiplicativity of $c_2(n)$ and $r(n)$, the bounds in \eqref{fractcoeffic} and that $X^2=o(T)$, we can show that
		\begin{align}\label{easybound}
			\begin{split}
		\mathscr{U}_2
		&=(1+o(1))\sum_{k\leq X^2}k^{-2\sigma}\sum_{\substack{m,n\geq1\\mn=k}}|c_2(m)|^2|r(n)|^2n\\
		&\leq(1+o(1))\sum_{n\geq1}\frac{|r(n)|^2}{n^{2\sigma-1}}\sum_{m\geq 1}\frac{|c_2(m)|^2}{m^{2\sigma}}\\
		&=(1+o(1))\mathscr{R}(\sigma)\prod_{p}\parentheses*{\sum_{\nu\geq0}\frac{|c_2(p^\nu)|^2}{p^{2\nu\sigma}}}\\
		&\ll\mathscr{R}(\sigma)\prod_{p}\parentheses*{1+\frac{|c_2(p)|^2}{p}}.
		\end{split}
		\end{align}
		
		Moving on estimating $\mathscr{U}_1$, we observe  that
		$
		\mathscr{U_1}=(1+o(1))M+O(E),
		$
		where
		\begin{align*}
			M:=\sum_{\substack{m_1,m_2,n_1,n_2\geq1\\m_1 n_1=m_2n_2}}\frac{c_1(m_1)\overline{c_1(m_2)}}{(m_1m_2)^\sigma}\frac{r(n_1)\overline{r(n_2)}}{(n_1n_2)^{\sigma-1/2}}
			\end{align*}
			and\begin{align*}
			E&:=\sum_{\substack{m_1,m_2,n_1,n_2\geq1\\m_1 n_1=m_2n_2\\\exists j\leq 2:\,n_j>X\text{ or }m_j>X}}\frac{\abs*{c_1(m_1){c_1(m_2)}r(n_1){r(n_2)}}}{(m_1m_2)^\sigma(n_1n_2)^{\sigma-1/2}}\\
			&\leq X^{-\eta}{\sum_{\substack{m_1,m_2,n_1,n_2\geq1\\m_1 n_1=m_2n_2}}}\frac{\abs*{c_1(m_1){c_1(m_2)}r(n_1){r(n_2)}}(m_1n_1m_2n_2)^{\eta}}{(m_1m_2)^\sigma(n_1n_2)^{\sigma-1/2}},
		\end{align*}
		with the last relation following by Rankin's trick.
		We then employ the multiplicativity of $r(n)$, which is supported in square-free integers, to rewrite the main term $M$ as
		\begin{align*}
			M&=	\sum_{m_1,m_2\geq1}\frac{c_1(m_1)\overline{c_1(m_2)}}{(m_1m_2)^\sigma}\frac{r\parentheses*{\frac{m_2}{(m_1,m_2)}}\overline{r\parentheses*{\frac{m_1}{(m_1,m_2)}}}}{\parentheses*{\frac{m_1m_2}{(m_1,m_2)^2}}^{\sigma-1/2}}\sum_{\substack{j\geq1\\\parentheses*{j,\frac{m_1}{(m_1,m_2)}}=\parentheses*{j,\frac{m_2}{(m_1,m_2)}}=1}}\frac{|r(j)|^2}{j^{2\sigma-1}}\\
			&=\mathscr{R}(\sigma)	\sum_{m_1,m_2\geq1}A(m_1)\overline{A(m_2)}B\parentheses*{\frac{m_2}{(m_1,m_2)}}\overline{B\parentheses*{\frac{m_1}{(m_1,m_2)}}},
		\end{align*}
		where
		\[
		A(m):=\frac{c_1(m)}{m^\sigma}\quad\text{ and }\quad B(m):=\frac{r(m)}{m^{\sigma-1/2}\prod_{p\mid m}\parentheses*{1+\frac{|r(p)|^2}{p^{2\sigma-1}}}}.
		\]
		Since both $A(m)$ and $B(m)$ are multiplicative functions and $B(m)$ is supported on square-free integers, we have that
			\begin{align*}
			M
			&=\mathscr{R}(\sigma)\prod_{p}\sum_{\mu,\nu\geq0}A(p^\nu)\overline{A(p^\mu)}B\parentheses{p^{\mu-\min\set{\mu,\nu}}}\overline{B\parentheses{p^{\nu-\min\set{\mu,\nu}}}}\\
			&=\mathscr{R}(\sigma)\prod_{p}\parentheses*{1+2\mathrm{Re}(A(p)\overline{B(p)})+\sum_{\nu\geq1}\parentheses*{|A(p^\nu)|^2+2\mathrm{Re}\parentheses*{B(p)A(p^\nu)\overline{A(p^{\nu+1})}}}}.
		\end{align*}
		Note that 
		\[
		A(p^\nu)\ll p^{\nu(\theta_\pi-\sigma)}=o(1)\quad\text{ and }\quad B(p)\ll\mathscr{L}\sum_{h\in\pi}\frac{|a_{L_h}(p)|}{p^{\sigma}\log p}\ll\frac{\mathscr{L}}{p^{\sigma}(\log p)^{\epsilon}}=o(1)
		\]
		 as follows from  \eqref{fractcoeffic} and the definition of $r(p)$, $p\in\mathcal{P}$.
		Therefore, we can derive that
		\[
	\sum_{p}\sum_{\nu\geq2}\parentheses*{|A(p^\nu)|^2+2\mathrm{Re}\parentheses*{B(p)A(p^\nu)\overline{A(p^{\nu+1})}}}=o(1).
		\]
	Moreover, $A(p)\ll(\log p)^{1-\epsilon}p^{-\sigma}$, $p\in\mathcal{P}$, whence
		\[
		\sum_{p}{A(p^2)A(p)B(p)}\ll\sum_{p\in\mathcal{P}}\frac{p^{2{(\theta_\pi-\sigma)}}\mathscr{L}}{p^{2\sigma}(\log p)^{2\epsilon-1}}\ll\mathscr{L}\sum_{p\geq\mathscr{L}^2}\frac{1}{p^{3/2}}=o(1).
		\]
		Thus,
		\begin{align*}
			M&=(1+o(1))\mathscr{R}(\sigma)\prod_{p}\parentheses*{1+2\mathrm{Re}(A(p)\overline{B(p)})+|A(p)|^2}\\
			&=(1+o(1))\mathscr{R}(\sigma)\exp\parentheses*{2\mathscr{F}_{\pi}(\sigma)}\prod_{p}\parentheses*{1+\frac{|A(p)|^2}{1+2\mathrm{Re}(A(p)\overline{B(p)})}}.
		\end{align*}
		
		To estimate the error term $E$ we argue exactly as we did for $M$ by keeping in mind that $p^\eta=O(1)$,  $p\in\mathcal{P}$.
		In particular, if we set
		\[
		\tilde{A}(m):=\frac{|c_1(m)|m^\eta}{m^\sigma}\quad\text{ and }\quad \tilde{B}(m):=\frac{|r(m)|m^\eta}{m^{\sigma-1/2}\prod_{p\mid m}\parentheses*{1+\frac{|r(p)|^2p^\eta}{p^{2\sigma-1}}}}
		\]
then we can show that
		\begin{align*}
			 E&\ll X^{-\eta}\brackets*{\prod_{p}\parentheses*{1+\frac{|r(p)|^2p^\eta}{p^{2\sigma-1}}}}\brackets*{\prod_{p}\parentheses*{1+2\tilde{A}(p)\tilde{B}(p)}}\brackets*{\prod_{p}\parentheses*{1+\frac{\tilde{A}(p)^2}{1+2\tilde{A}(p)\tilde{B}(p)}}}\\
			&\ll\frac{\mathscr{R}(\sigma)}{X^\eta}\brackets*{\prod_{p}\parentheses*{1+{|r(p)|^2(p^{\eta}-1)}}}\brackets*{\prod_p\parentheses*{1+\frac{2|c_1(p)r(p)|p^{2\eta}}{\sqrt{p}}}}\prod_{p}\parentheses*{1+\frac{2|c_1(p)|^2p^{2\eta}}{p}}.
		\end{align*}
		In view of \ref{resonator properties(i)}, \ref{resonator properties(v)}, \eqref{fractcoeffic}, \ref{resonator properties(iv)} and \eqref{PNT} we derive that
		\[
		E\leq\mathscr{R}(\sigma)X^{-\eta}\exp\parentheses*{\eta\parentheses*{\frac{1}{\lambda}+o(1)}\log X}.
		\]
		Hence,
		\[
		\mathscr{U_1}=(1+o(1))\mathscr{R}(\sigma)\exp\parentheses*{2F_{\pi}(\sigma)}\prod_{p}\parentheses*{1+\frac{|A(p)|^2}{1+2\mathrm{Re}(A(p)\overline{B(p)})}}+o(1)\mathscr{R}(\sigma),
		\]
		which  together with \eqref{easybound} and \eqref{fractcoeffic} yields that
		\[
	\mathscr{U}\ll\sqrt{\mathscr{U}_1\mathscr{U}_2}\ll	\mathscr{V}_\pi(\sigma)\prod_{p\leq X}\parentheses*{1+(q^2+o(1))\frac{\abs*{\sum_{h\in\pi}a_{L_h}(p)}^2}{p}}^{1/2}.
		\]
		Relation \eqref{PNT} implies that the product on the right-hand side above is bounded by $(\log T)^{q^2\kappa_\pi/2+o(1)}$.
		The lemma follows by noting that $S_\pi(\sigma,T)\ll\mathscr{U}$ for $U=1$ and $V=2$ from the properties of $w(t,T)$.
	\end{proof}
	
	\begin{proof}[Proof of Propostion \ref{keyprop}]
		As we have already discussed in the beginning of the section, in view of \eqref{intermediate} and \eqref{mainasymp} it suffices to show that 
 $I_\pi\parentheses*{\textstyle{\frac{1}{2}},T}\gg	\mathscr{V}_\pi\parentheses*{\textstyle{\frac{1}{2}}}	(\log T)^{-A}$ for some $A>0$.
		
		Since
		\begin{align*}
			\abs*{s_\pi(s,T)}\ll \abs*{L_\pi(s)}^{q}+\abs*{d_\pi(s,T)}^{1/v}
			\quad\text{ and }\quad
			\abs*{d_\pi(s,T)}^{1/v}\ll\abs*{L_\pi(s)}^{q}+\abs*{s_\pi(s,T)},
		\end{align*}
		we have that
		\begin{align}\label{basicinequ}
			S_\pi(\sigma,T)\ll I_\pi(\sigma,T)+D_\pi(\sigma,T)\quad\text{ and }\quad
			D_\pi(\sigma,T)\ll  I_\pi(\sigma,T)+S_\pi(\sigma,T).
			\end{align}

		We then consider two cases. 
		If $D_\pi\parentheses*{\frac{1}{2},T}\leq \sqrt{\mathscr{V}_\pi\parentheses*{\textstyle{\frac{1}{2}}}}$, then \eqref{basicinequ} and Lemma \ref{lowerboundapprox1} readily imply that 
		$
		{\mathscr{V}_\pi\parentheses*{\textstyle{\frac{1}{2}}}}\ll	I_\pi\parentheses*{\textstyle{\frac{1}{2}},T}$.		
		If on the other hand $D_\pi\parentheses*{\frac{1}{2},T}>\sqrt{\mathscr{V}_\pi\parentheses*{\textstyle{\frac{1}{2}}}}$,
		then Lemma \ref{lowerboundapprox2} and \eqref{basicinequ} yield that
		\begin{align}\label{alternative}
			\begin{split}
				D_\pi(\sigma,T)
				&\ll D_\pi({\textstyle{\frac{1}{2}}},T)\mathscr{V}_\pi\parentheses*{\textstyle{\frac{1}{2}}}^{\frac{1/2-\sigma}{2\beta-1}}T^{(1/2-\sigma)\parentheses*{1-1/\beta}{\Delta}/v}+e^{-\epsilon T^{2}}\\
				&\ll \parentheses*{I_\pi({\textstyle{\frac{1}{2}}},T)+S_\pi({\textstyle{\frac{1}{2}}},T)}T^{(1/2-\sigma)\parentheses*{1-1/\beta}{\Delta}/v}+e^{-\epsilon T^{2}}.
			\end{split}
		\end{align}
		In combination with \eqref{basicinequ} and Lemma \ref{Gabriel} we derive that
		\begin{align}\label{prelude}
			\begin{split}
				S_\pi(\sigma,T)&\ll I_\pi({\textstyle{\frac{1}{2}}},T)^{1+\frac{1/2-\sigma}{2\beta-1}}+I_\pi({\textstyle{\frac{1}{2}}},T)T^{(1/2-\sigma)\parentheses*{1-1/\beta}{\Delta}/v}\\
				&\quad+S_\pi({\textstyle{\frac{1}{2}}},T)T^{(1/2-\sigma)\parentheses*{1-1/\beta}{\Delta}/v}+e^{-\epsilon T^{2}}
			\end{split}
		\end{align}
		uniformly for $1/2\leq\sigma\leq\beta$ and $T\gg1$. 
		We now show that for sufficiently large but fixed $B,\beta\gg1$, any $T\gg1$  and $\sigma_T:=1/2+B\log\log T/\log T$
		\begin{align}\label{interlude}
			S_\pi(\sigma_T, T)\geq S_\pi({\textstyle{\frac{1}{2}}},T)T^{(1/2-\sigma_T)\parentheses*{1-1/\beta}{\Delta}/v}\log\log T.
		\end{align}
		For, otherwise, Lemma \ref{lowerboundapprox1} and Lemma \ref{upperboundapprox1} would yield that
		\begin{align*}
			\frac{\mathscr{V}_\pi\parentheses*{\sigma_T}}{\mathscr{V}_\pi\parentheses*{\textstyle{\frac{1}{2}}}}
			&\ll\frac{S_\pi(\sigma_T,T)}{S_\pi({\textstyle{\frac{1}{2}}},T)}(\log T)^{q^2\kappa_\pi/2+o(1)}\ll T^{(1/2-\sigma_T)\parentheses*{1-1/\beta}{\Delta}/v}(\log T)^{q^2\kappa_\pi/2+o(1)},
		\end{align*}
		which implies that
		\begin{align}\label{Heath-Brown}
			\begin{split}
\log \mathscr{V}_\pi({\textstyle{\frac{1}{2}}})-\log \mathscr{V}_\pi(\sigma_T)\geq\parentheses*{\frac{{B\Delta}\parentheses*{1-1/\beta}}{v}-\frac{q^2\kappa_\pi}{2}+o(1)}\log\log T.
			\end{split}
		\end{align}
		On the other hand, in view of  \ref{resonator properties(i)}, \ref{resonator properties(v)} and \ref{resonator properties(vi)} we also have that
		\begin{align}\label{difference}
			\begin{split}
		\log \mathscr{V}_\pi({\textstyle{\frac{1}{2}}})-\log \mathscr{V}_\pi(\sigma_T)
				&=\log\mathscr{R}\parentheses*{\textstyle{\frac{1}{2}}}-\log\mathscr{R}(\sigma_T)+\mathscr{F}_{\pi}\parentheses*{\textstyle{\frac{1}{2}}}-\mathscr{F}_{\pi}(\sigma_T)\\
				&=-\sum_{p}\log\parentheses*{1-\frac{|r(p)|^2(1-p^{1-2\sigma_T})}{1+|r(p)|^2}}+o(1)\\
				&=	(1+o(1)){\sum_p|r(p)|^2(1-p^{1-2\sigma_T})}+o(1)\\
				&=\parentheses*{\frac{B\Delta}{\lambda}+o(1)}\log\log T
			\end{split}
		\end{align}
		which together with $\lambda>v$ contradicts \eqref{Heath-Brown} for sufficiently large but fixed $B,\beta$ and any sufficiently large $T\gg1$. 
		Hence, relation \eqref{interlude} must hold (for such $B$ and $\beta$) and \eqref{prelude} becomes
		\begin{align*}
			S_\pi(\sigma_T,T)
			&\ll I_\pi({\textstyle{\frac{1}{2}}},T)^{1+\frac{1/2-\sigma_T}{2\beta-1}}+I_\pi({\textstyle{\frac{1}{2}}},T)T^{(1/2-\sigma_T)\parentheses*{1-1/\beta}{\Delta}/v}+e^{-\epsilon T^2}\\
			&\ll I_\pi({\textstyle{\frac{1}{2}}},T)\parentheses*{ I_\pi({\textstyle{\frac{1}{2}}},T)^{-\frac{B\log\log T}{(2\beta-1)\log T}}+1}+e^{-\epsilon T^2},
		\end{align*}
		which in combination with \eqref{difference} and Lemma \ref{lowerboundapprox1} yields that
		\[
		(\log T)^{-B\Delta/\lambda+o(1)}\mathscr{V}_\pi({\textstyle{\frac{1}{2}}})= \mathscr{V}_\pi(\sigma_T)\ll S_\pi(\sigma_T,T)\ll I_\pi({\textstyle{\frac{1}{2}}},T).
		\]
		
		Thus, regardless of the size of $D_\pi\parentheses*{\frac{1}{2},T}$, we deduce that
		\begin{align*}
			I_{\pi}\parentheses*{{\textstyle{\frac{1}{2}}},T}\geq\exp\parentheses*{\parentheses{q\kappa_\pi+o(1)}\sqrt{\frac{\log X}{\kappa_L\lambda\log\log X}}}\prod_{p}\parentheses*{1+|r(p)|^2}.
		\end{align*}
	\end{proof}
	\section{Extreme values off the critical line: Proof of Theorem \ref{thm:ourresult1}}\label{secondthm}
	For the proof of Theorem \ref{thm:ourresult1} we will make use of short Dirichlet polynomial approximations of the involved $L$-functions  twisted by a suitable resonating polynomial.
	The existence of such approximations relies on the observation that if an element of $\mathcal{S}$ is zero free in a horizontal strip, then it satisfies therein a bound to the effect of the Lindel\"of Hypothesis.
	It will be practical to distinguish these zero-free strips by means of the following definition. 
	
	\begin{Def}
		For $\alpha\in(0,1)$, $T\gg1$ and an arbitrary number $A\in[5T/4,7T/4]$, let
		\[
		\mathcal{I}=\mathcal{I}(\alpha,T,A):=[A+T^\alpha/2,A+5T^\alpha/2]\subseteq[T,2T].
		\]
		If $L\in\mathcal{S}$ is zero free in $\set{s:\sigma\geq\sigma_0,\,t\in I}$, then we say that $\mathcal{I}$ {\it is $\sigma_0$-good for $L(s)$}. 
	\end{Def}
	The first lemma is a modified version of Littlewood's bound for $\log\zeta(s)$ to the right of the critical line under the RH \cite[Theorem 14.2]{Titchmarsh1986}.
	Although its proof follows with only few adjustments, we repeat it here for the sake of completeness.
	\begin{lem}\label{Lindeloefinstrips}
		Let $L\in\mathcal{S}$.
		For any $T\gg1$ and in any interval $\mathcal{I}$ which is $\sigma_0$-good for $L(s)$, we have that
		\[
		\log L(\sigma+it)\ll(\log T)^{(1-\sigma)/(1-\sigma_0)+\epsilon}
		\]
		uniformly for $\sigma_0+\epsilon\leq\sigma\leq 2$ and $t\in [A+3T^\alpha/4,A+9T^\alpha/4]=:\mathcal{I}_\mu\subseteq\mathcal{I}$.
	\end{lem}
	\begin{proof}
		By assumption $\log L(s)$ is analytic in $\set{s:\sigma>\sigma_0,\,t\in\mathcal{I}^o}\subset\Omega_L$.
		Thus, if we let $\tau\in\mathcal{J}:=[A+3T^\alpha/4-2\log\log T,A+9T^\alpha/4+2\log\log T]\subset\mathcal{I}$ and $D_1(\tau)$ and $D_2(\tau)$ denote two closed disks of center $2+i\tau$ and radii $2-\sigma_0-\epsilon$  and $2-\sigma_0-\epsilon/2$, respectively, then $D_1(\tau),D_2(\tau)\subset\Omega_L$ and, thus, the Borel-Carath\'eodory theorem, the absolute convergence of the Euler product of $L(s)$ for $\sigma>1$ and the polynomial growth \eqref{ordergrowth} of $L(s)$  yield that
		\begin{align}\label{Borel-Cara}
			\begin{split}
				\max_{s\in D_1(\tau)}|\log L(s)|
				&\leq\frac{4-2\sigma_0-2\epsilon}{{\epsilon}/{2}}\max_{s\in D_2(\tau)}\log|L(s)|+\frac{4-2\sigma_0-{3\epsilon}/{2}}{{\epsilon}/{2}}|\log L(2+it)|\ll\log T.
			\end{split}
		\end{align}
		If now $t \in\mathcal{I}_\mu$, we apply the Hadamard three circle theorem in the circles $C_1,C_2,C_3\subset\Omega_L$ of center $\log\log T+it=:\sigma_1+it$ and radii $\sigma_1-1-\epsilon$, $\sigma_1-\sigma$ and $\sigma_1-\sigma_0-\epsilon$, respectively,
		\begin{align}\label{Hadamard3}
			\max_{s\in C_2}|\log L(s)|\ll\parentheses*{\max_{s\in C_3}|\log L(s)|}^{a}\parentheses*{\max_{s\in C_1}|\log L(s)|}^{1-a},
		\end{align}
		where
		\begin{align}\label{Hadamardexp}
			a=\parentheses*{\log\frac{\sigma_1-\sigma}{\sigma_1-1-\epsilon}}\parentheses*{\log\frac{\sigma_1-\sigma_0-\epsilon}{\sigma_1-1-\epsilon}}^{-1}=\frac{1-\sigma}{1-\sigma_0}+O\parentheses*{\epsilon+\frac{1}{\sigma_1}}.
		\end{align}
		Observe that the sets $C_1$ and $C_3\cap\set{z:\mathrm{Re} z\geq1+\epsilon}$ lie in the half-plane of absolute convergence of the Euler product of $ L(s)$.
		Hence, 
		\[
		\max_{s\in C_1}|\log L(s)|\ll1\quad\text{ and }\quad
		\max_{s\in C_3\cap\set{z:\mathrm{Re} z\geq1+\epsilon}}|\log L(s)|\ll1.
		\]
		Moreover, if $x+iy\in C_3\cap\set*{z:\sigma_0+{\epsilon}\leq\mathrm{Re} z\leq 1+\epsilon}$, then $y\in\mathcal{J}$ and  $x+iy\in D_1(y)$.
		Thus, in view of \eqref{Borel-Cara} we derive that
		\[
		\max_{s\in C_3\cap\set*{z:\sigma_0+{\epsilon}\leq\mathrm{Re} z\leq 1+\epsilon}}|\log L(s)|\ll\log T.
		\]
		The lemma follows now from  relations \eqref{Hadamard3}, \eqref{Hadamardexp} and the above bounds.
	\end{proof}
	We introduce next the Dirichlet polynomial that will resonate with the $L$-functions $L_1,\dots,L_{H+G}$ from Theorem \ref{thm:ourresult1}, namely elements of $\mathcal{S}_{\text{poly}}$ which satisfy the Ramanujan Hypothesis and \eqref{PNT0}.
	Recall in particular that $a_{L_h}(p)\ll1$ for any prime $p$ and $h\leq H+G$.
		Perhaps the same resonator as on the critical line could be implemented here but it would make the proof more technical without any real gain other than a larger constant $D$ in the exponential.
		We opted for a simpler resonator as devised by Voronin \cite[VIII, \S2]{Karatsuba1992} and (independently) Hilberdink \cite{Hilberdink2009}.
	Let $0<\beta<{\alpha}/{2}$ and
	\[
	R(s):=\prod_{p}(1+r(p)p^{-s})=\sum_{n\geq 1}r(n)n^{-s}
	\] 
	with $r(n)$ being a multiplicative function supported on square-free integers and defined by
	\begin{align*}
		Cr(p):=\sum_{h\leq H}{a_{L_h}(p)}-\sum_{H<h\leq H+G}{{a_{L_h}(p)}},\quad p\in\mathcal{P}:=\set{x_\epsilon\leq p\leq\beta\log T:p\text{ prime}},
	\end{align*}
	and by $0$ for any other prime $p$.
	The constant $C\gg H+G$ will be chosen sufficiently large later on.
	Moreover, $x_\epsilon$ is fixed large enough so that
	\begin{align}\label{orthonormality}
		\max_{h\leq H+G}\set*{\abs*{\sum_{p\leq x}|a_{L_h}(p)|^2-\frac{\kappa_hx}{\log x}},\abs*{\sum_{h\neq g\leq H+G}\sum_{p\leq x}a_{L_h}(p)a_{L_g}(p)}}\leq \frac{\epsilon x}{\log x},\quad x\geq x_\epsilon,
	\end{align}
	as follows from \eqref{PNT0} with $\kappa_h:=\kappa(h,h)>0$.
	We can then prove the following lemma.
	\begin{lem}\label{computing}
		Let $\sigma\in\parentheses*{0,1}$.
		There are $B_h>0$, $h\leq H+G$, such that for any $T\gg1$
		\[
		(-1)^{\mathrm{j}(h)}\sum_{p}\frac{\mathrm{Re}\parentheses{\overline{r(p)}a_{L_h}(p)}}{(1+|r(p)|^2)p^\sigma}\leq -B_h\frac{(\log T)^{1-\sigma}}{\log\log T},
		\]
		where $\mathrm{j}(h)=1$ if $h\leq H$ and $\mathrm{j}(h)=0$ if $H<h\leq H+G$.
		For $\sigma=1$, the right-hand side of the above relation ought to be replaced by $B_h\log\log\log T$.
	\end{lem}
	\begin{proof}
		Assume first that $h\leq H$.
		By the definition of $r(p)\ll\frac{1}{C}$, we see that
		\begin{align}\label{splitting}
			\begin{split}
				C\sum_{p}\frac{\mathrm{Re}\parentheses{\overline{r(p)}a_{L_h}(p)}}{(1+|r(p)|^2)p^\sigma}
				&=\sum_{p\in\mathcal{P}}\frac{|a_{L_h}(p)|^2}{(1+|r(p)|^2)p^\sigma}+\mathrm{Re}\parentheses*{\sum_{h\neq g\leq H}\sum_{p\in\mathcal{P}}\frac{a_{L_h}(p)\overline{a_{L_g}(p)}}{(1+|r(p)|^2)p^\sigma}}\\
				&\quad-\mathrm{Re}\parentheses*{\sum_{H< g\leq H+G}\sum_{p\in\mathcal{P}}\frac{a_{L_h}(p)\overline{a_{L_g}(p)}}{(1+|r(p)|^2)p^\sigma}}.
			\end{split}
		\end{align}
		In view of \eqref{orthonormality} and partial summation, the first sum in the right-hand side above equals to
		\begin{align}\label{Selberg1}
			\begin{split}
				\parentheses*{1+O\parentheses*{\frac{1}{C^2}}}\sum_{p\in\mathcal{P}}\frac{|a_{L_h}(p)|^2}{p^\sigma}&=\parentheses*{1+O\parentheses*{\frac{1}{C^2}}}\frac{(\kappa_h+O(\epsilon))(\beta\log T)^{1-\sigma}}{\log\log T}.
			\end{split}
		\end{align}
		Similarly, the first double sum in the right-hand side of \eqref{splitting} can be seen to be
		\begin{align}\label{Selberg2}
			\begin{split}
				\sum_{h\neq g\leq H}\sum_{p\in\mathcal{P}}\frac{a_{L_h}(p)\overline{a_{L_g}(p)}}{p^\sigma}\parentheses*{1+O(|r(p)|^2)}
				&\ll\frac{\epsilon(\beta\log T)^{1-\sigma}}{\log\log T}+\frac{1}{C^2}\sum_{p\leq\beta\log T}\frac{1}{p^\sigma}\\
				&\ll\parentheses*{\epsilon+\frac{1}{C^2}}\frac{(\beta\log T)^{1-\sigma}}{\log\log T},
			\end{split}
		\end{align}
		while the second double sum in the right-hand side of \eqref{splitting} is bounded as above for the same reason.
		The first part of the lemma for $h\leq H$ follows now from  \eqref{splitting}--\eqref{Selberg2} by fixing sufficiently large $C\gg 1$.
		If $H<h\leq H+G$, we obtain completely analogously the same result because the main contribution will still be coming from a sum of the shape of \eqref{Selberg1}.
		The case of $\sigma=1$ is treated in a similar manner.
	\end{proof}
	
	\begin{prop}\label{technical}
		Let $L,\tilde{L}\in\mathcal{S}_{\text{poly}}$ satisfy the Ramanujan Hypothesis and $\sigma\in(\sigma_0,1]$.
		For any $T\gg1$ and in any interval $\mathcal{I}$ which is $\sigma_0$-good for $L\tilde{L}(s)$, we have that
		\begin{align*}
			\mathscr{I}:=\frac{1}{T^\alpha}\int_{A+T^\alpha}^{A+2T^\alpha}\abs*{\frac{L(\sigma+it)}{\tilde{L}(\sigma+it)}R(it)}^2\mathrm{d}t\ll{\prod_p\parentheses*{1+|r(p)|^2+{2p^{-\sigma}\mathrm{Re}\parentheses{\overline{r(p)}a_{L/\tilde{L}}(p)}}}}.
		\end{align*}
	\end{prop}
	\begin{proof}
		Define first the function
		\[
		F_\sigma(z):=	\frac{L(z)}{{\tilde{L}(z)}}R(z-\sigma)={\sum_{n\geq1}{a_{L/{\tilde{L}}}(m)}{m^{-z}}}{\sum_{n\geq1}{r(n)n^\sigma}{ n^{-z}}}=:\sum_{n\geq1}{f_\sigma(n)}{n^{-z}},\quad\mathrm{Re}z>1.
		\]
		Then  Perron's formula  \cite[Lemma 3.12]{Titchmarsh1986} yields that
		\[
		\sum_{n\leq Y}f_\sigma(n){n^{-s}}=\frac{1}{2\pi i}\int_{\frac{1}{2}-iY}^{\frac{1}{2}+iY}F_\sigma(s+w)\frac{Y^w}{w}\mathrm{d}w+O\parentheses*{\frac{1}{\sqrt{Y}}},\quad t\in[A+T^\alpha,A+2T^\alpha],
		\]
		where $Y:= T^\alpha/5$ is, without loss of generality, half an odd integer.
		Observe that $F(s+w)$ can be continued analytically as a function of $w$ inside the rectangle $\sigma_0-\sigma<\mathrm{Re}w<1$, $|\mathrm{Im}w|<Y$.
		Hence, by moving the path of integration to the line segment with endpoints $-\epsilon\pm iX$ we pick up only a simple pole at $w=0$ of residue $F_\sigma(s)$ and we obtain that
		\begin{align*}
			F_\sigma(s)-\sum_{n\leq Y}f_\sigma(n){n^{-s}}+o(1)
			&=\frac{1}{2\pi i}\braces*{\int_{\frac{1}{2}+iY}^{-\epsilon+iY}+\int^{\frac{1}{2}-iY}_{-\epsilon-iY}+\int_{-\epsilon+iY}^{-\epsilon-iY}}F_\sigma(s+w)\frac{Y^w}{w}\mathrm{d}w\\
			&\ll\max\set*{\abs{F_\sigma(u+iv)}:\sigma-\epsilon\leq u\leq\sigma+\frac{1}{2},\,v\in\mathcal{I}_\mu},
		\end{align*}
		where $\mathcal{I}_\mu$ is as in Lemma \ref{Lindeloefinstrips}.
		Since the interval $\mathcal{I}$ is $\sigma_0$-good for $L$ and ${\tilde{L}}$, we can  employ the aforementioned lemma to show that
		\[
		\max\set*{\abs*{\frac{L(u+iv)}{{\tilde{L}(u+iv)}}}:\sigma-\epsilon\leq u\leq\sigma+\frac{1}{2},\,v\in\mathcal{I}_\mu}\leq T^{o(1)}.
		\]
		Moreover $N:=\max\mathrm{supp}(r)\ll T^{\beta+o(1)}$ by the prime number theorem and, therefore, we have  uniformly for $u\geq \sigma-\epsilon$ and $t\in\mathbb{R}$ that
		\[
		R(u+iv-\sigma)\ll\sum_{n\leq N}|r(n)|n^{\epsilon}\ll T^{(\beta+o(1))(1+\epsilon)}.
		\]
		Hence,
		\[
		\frac{L(\sigma+it)}{\tilde{L}(\sigma+it)}R(it)=F_\sigma(\sigma+it)=\sum_{n\leq Y}f_\sigma(n)n^{-\sigma-it}+O\parentheses*{T^{\beta+\epsilon}},\quad t\in[A+T^\alpha,A+2T^\alpha].
		\]
		We then insert the above Dirichlet polynomial approximation in $\mathscr{I}$, apply the Cauchy-Schwarz inequality to separate the main term from the error term, apply further Lemma \ref{M-V}  to the main term and lastly appeal to the multiplicativity of $f_\sigma(n)$, to conclude that
		\begin{align}\label{multiprod}
			\mathscr{I}\ll\sum_{n\leq T^{\alpha}}\frac{|f_\sigma(n)|^2}{n^{2\sigma}}\ll\prod_{p\leq T^{\alpha}}\parentheses*{1+\frac{|f_\sigma(p)|^2}{p^{2\sigma}}+\sum_{\nu\geq2}\frac{|f_\sigma(p^\nu)|^2}{p^{2\nu\sigma}}}.
		\end{align}
		Observe that for any prime $p$ and integer $\nu\geq1$
		\[
		f_\sigma(p^\nu)=\sum_{\substack{n_1,n_2\geq1\\n_1n_2=p^\nu}}a_{L/\tilde{L}}(n_1)r(n_2)n_2^\sigma=a_{L/\tilde{L}}(p^\nu)+a_{L/\tilde{L}}(p^{\nu-1})r(p)p^\sigma\ll p^{(\nu-1)\epsilon+\sigma}
		\]
		and, in particular,
		\[
		|f_\sigma(p)|^2\leq2|a_{L}(p)|^2+2|a_{\tilde{L}}(p)|^2+|r(p)|^2p^{2\sigma}+2p^\sigma\mathrm{Re}\parentheses{\overline{r(p)}a_{L/\tilde{L}}(p)}.
		\]
		Since $W(p):=1+|r(p)|^2+2p^{-\sigma}\mathrm{Re}\parentheses{\overline{r(p)}a_{L/\tilde{L}}(p)}> \frac{1}{2}$ and
		\[
		\sum_{p\geq1}\frac{|a_L(p)|^2+|a_{\tilde{L}}(p)|^2}{p^{2\sigma}}+\sum_{p\geq1}\sum_{\nu\geq2}p^{2(\nu-1)(\epsilon-\sigma)}=O(1),
		\]
		we may take a factor $W(p)$ out from each factor of the product in \eqref{multiprod} and derive the statement of the lemma.
	\end{proof}
	We are now in position to prove Theorem \ref{thm:ourresult1}.
	\begin{proof}[Proof of Theorem \ref{thm:ourresult1}]
		Let $\sigma\in(\sigma_0,1]$ be fixed and 
		\begin{align*}
			Z:=\left\{
			\begin{array}{ll}
				{D}\frac{(\log T)^{1-\sigma}}{\log\log T},&\sigma\neq1,\\
				D\log\log\log T,&\sigma=1,
			\end{array}
			\right.
		\end{align*}
		where $D:=\min_{h\leq H+G}B_h$.
		Since $\min_{h\leq H+G}N_{L_h}(T)\gg T\log T$ and $\sum_{h\leq H+G}N_{L_h}(T)=o(T^{1-\alpha})$, for any $T\gg1$ there is (at least one) interval $\mathcal{I}\subseteq[T,2T]$ which is $\sigma_0$-good for any $L_h$, $h\leq H+G$.
		The theorem will then follow if we can show that there is $t\in [A+T^\alpha, A+2T^\alpha]$ such that
		\begin{align*}
			\sum_{h\leq H}\abs*{{L_h}(\sigma+it)}^{-2}+\sum_{H\leq h\leq H+G}\abs*{{L_{h}}(\sigma+it)}^{2}\leq \exp\parentheses*{-2Z+O(1)}.
		\end{align*}
		To that end, it suffices to prove that
		\begin{align*}
			\int_{A+T^\alpha}^{A+2T^\alpha}&\parentheses*{\sum_{h\leq H}\abs*{{L_h}(\sigma+it)}^{-2}+\sum_{H\leq h\leq H+G}\abs*{{L_{h}}(\sigma+it)}^{2}}|R(it)|^2\mathrm{d}t
			\\
			&\qquad\qquad\qquad\qquad\qquad\qquad\qquad\qquad\leq	\exp\parentheses*{-2Z+O(1)}\int_{A+T^\alpha}^{A+2T^\alpha}|R(it)|^2\mathrm{d}t
		\end{align*}
		or, equivalently, that for every $h\leq H+G$,
		\begin{align}\label{technical1}
			\int_{A+T^\alpha}^{A+2T^\alpha}\abs*{L_h^{(-1)^{\mathrm{j}(h)}}(\sigma+it)}^{2}|R(it)|^2\mathrm{d}t\leq \exp\parentheses*{-2Z+O(1)}	\int_{A+T^\alpha}^{A+2T^\alpha}|R(it)|^2\mathrm{d}t,
		\end{align}
			where $\mathrm{j}(h)=1$ if $h\leq H$ and $\mathrm{j}(h)=0$ if $H<h\leq H+G$.
		By Lemma \ref{M-V} we have that
		\begin{align*}
			\frac{1}{T^\alpha}\int_{A+T^\alpha}^{A+2T^\alpha}|R(it)|^2\mathrm{d}t
			=\parentheses*{1+O\parentheses*{T^{\beta-\alpha+\epsilon}}}\sum_{n\leq N}|r(n)|^2
			\asymp\prod_p\parentheses*{1+|r(p)|^2}.
		\end{align*}
		Therefore, in view of Proposition \ref{technical} we obtain that
		\[
		\frac{\int_{A+T^\alpha}^{A+2T^\alpha}\abs*{L_h^{\mathrm{j}(h)}(\sigma+it)}^{2}|R(it)|^2\mathrm{d}t}{\int_{A+T^\alpha}^{A+2T^\alpha}|R(it)|^2\mathrm{d}t}\ll\prod_p\parentheses*{1+	(-1)^{\mathrm{j}(h)}\frac{\mathrm{Re}\parentheses{\overline{r(p)}a_{L_h}(p)}}{(1+|r(p)|^2)p^\sigma}}.
		\]
		Since $|r(p)a_{L_h}(p)|=o(p^\sigma)$, we see that inequality \eqref{technical1} follows from the above relation and Lemma \ref{computing}.
		This concludes the proof of the theorem.
	\end{proof}
	
	\section{Proof of Theorem \ref{thm:ourresult2}}\label{thirdthm}
	The proof of Theorem \ref{thm:ourresult2} relies on the next two lemmas whose proofs are given in \cite[Corolllary 4.2]{PANKOWSKI2013} and \cite[Theorem 1 (i)]{Chen2000}, respectively.
	For the latter result in particular, it is worth mentioning that it is the culmination of research originating from the trigonometric proof of Bohr and Jessen \cite{Bohr1932} on Kronecker's theorem, which could be regarded as the predecessor of the resonance method.
	\begin{lem}\label{convol}
		Let $L\in\mathcal{S}$ and suppose that, for $t_0\geq15$ and $\tau=\tau(t_0)=O(t_0)$, $L(s)$ does not vanish in $\set{s:\sigma\geq\sigma_0,\,|t-t_0|\leq2\tau}$.
		Then for any $\phi\in\mathbb{R}$ and $x>0$ we have that
		\[
		\max_{|t|\leq\tau}\mathrm{Re}\parentheses{e^{-i\phi}\log L(s_0+it)}\geq\frac{1}{2}\mathrm{Re}\parentheses*{\sum_{\abs*{\log\frac{n}{x}}<1}\frac{b_L(n)}{n^{s_0}}\parentheses*{1-\abs*{\log\frac{n}{x}}}}+O\parentheses*{x\frac{\tau+\log t_0}{\tau^2}}.
		\]
	\end{lem}
	
	\begin{lem}\label{Chen}
		Let $\lambda_1,\dots,\lambda_N$ be real numbers linearly independent over $\mathbb{Q}$.
		Then for all real numbers $\eta_1,\dots,\eta_N$ and for all positive real numbers $\xi_1,\dots,\xi_N$ and $T_1<T_2$, we have that
		\[
		\inf_{t\in[T_1,T_2]}\sum_{n\leq N}\xi_n\norm{\lambda_n t-\eta_n}^2\leq\frac{1}{4}\sum_{n\leq N}\xi_j\parentheses*{\sin^2\parentheses*{\frac{\pi}{2(M+1)}}+\frac{M^N}{\pi(T_2-T_1)\Lambda}},
		\]
		where
		\[
	\Lambda:=\min\set*{\abs*{\sum\nolimits_{n\leq N}u_n\lambda_n}\neq0:u_n\in\mathbb{Z},\,|u_n|\leq M}.
		\]
	\end{lem}
	\begin{proof}[Proof of Theorem \ref{thm:ourresult2}]
		By assumption, for any $T\gg1$ there is (at least one) interval $\mathcal{I}\subseteq[T,2T]$ which is $\sigma_0$-good for all $L_h(s)$, $h\leq H$.
		Hence, if we set $\tau:=(\log T)^{(1+\sigma_0)/2}\sqrt{\log\log T}$ and $x_h:={e^{2h}(\log T)}/{(CM)}$, $h\leq H$, where $C,M$ will be determined later on and any subsequent implicit constants are independent of $M$, then Lemma \ref{convol} yields that
		\begin{align}\label{Mont}
			\begin{split}
				\max_{|t|\leq\tau}&\mathrm{Re}\parentheses{e^{-i\phi_h}\log L_h(s_0+it)}\\
				&\geq\frac{1}{2}\sum_{\abs*{\log\frac{p}{x_h}}<1}\frac{|a_{L_h}(p)|\cos(\arg(a_{L_h}(p))-t_0\log p)}{p^{\sigma_0}}\parentheses*{1-\abs*{\log\frac{p}{x_h}}}\\
				&\quad+O\parentheses*{\sum_{\ell\geq2}\sum_{\frac{x_h}{e}<p^\ell<ex_h}\frac{|b_L(p^\ell)|}{p^{\ell\sigma}}+\frac{(\log T)^{1-\sigma_0}}{M\log\log T}}
			\end{split}
		\end{align}
		for any $t_0\in[A+T^\alpha,A+2T^\alpha]$.
		Moreover, if $\theta:=\max\{\theta_{L_h}:h\leq H\}$, then the double sum in the error term can be seen \cite[Footnote 1]{Aistleitner2017} to be bounded by $x_h^{2\theta-2\sigma+1/2}\log x=o\parentheses*{(\log T)^{1-\sigma_0}/\log\log T}$.
		
		Let now
		\[
		\set{\lambda_1,\dots,\lambda_N}:=\bigsqcup_{h\leq H}\set*{\frac{\log p}{2\pi}:\frac{x_h}{e}<p<ex_h},
		\]
		where we note that the sets on the right-hand side are disjoint from our choices of $x_h$.
		This implies that the numbers $\lambda_n$, $n\leq N$, are linearly independent over $\mathbb{Q}$.
		Hence, if we set $\eta_n=\arg(a_{L_h}(p))/(2\pi)$ and $\xi_n=|a_{L_h}(p)|/{p^{\sigma_0}}$ whenever $\lambda_n=(\log p)/(2\pi)$, then Lemma \ref{Chen} yields the existence of a $t_0\in[A+T^\alpha,A+2T^\alpha]$ such that
		\begin{align*}
			\mathscr{N}(t_0)
			&:=\sum_{h\leq H}\sum_{\abs*{\log\frac{p}{x_h}}<1}\frac{|a_{L_h}(p)|}{p^{\sigma_0}}\norm*{\frac{\arg(a_{L_h}(p)-t_0\log p)}{2\pi}}^2\\
			&\ll\sum_{h\leq H}\sum_{\abs*{\log\frac{p}{x_h}}<1}\frac{|a_{L_h}(p)|}{p^{\sigma_0}}\parentheses*{\frac{1}{M^2}+\frac{M^N}{T^\alpha\Lambda}},
		\end{align*}
		where 
		\[
		\Lambda=\min\set*{\abs*{\sum\nolimits_{\frac{x_1}{e}<p<ex_H}u_p{\log p}}\neq0:u_p\in\mathbb{Z},\,|u_p|\leq M}.
		\]
		The above quantity can be seen \cite[\S8.8, Lemma $\beta$]{Titchmarsh1986} to satisfy
		$\Lambda>(ex_H)^{-MN}$.
		Taking also into account the prime number theorem and that $N<\pi(ex_H)$, we obtain that
		\[
		\mathscr{N}(t_0)\ll\sum_{h\leq H}\sum_{\abs*{\log\frac{p}{x_h}}<1}\frac{|a_{L_h}(p)|}{p^{\sigma_0}}\parentheses*{\frac{1}{M^2}+\exp\parentheses*{\frac{2ex_H\log M}{\log(ex_H)}+2Mex_H-\alpha\log T}}.
		\]
		Thus, if we fix $C=C(\alpha)>0$ sufficiently large, then for any $M\geq1$ and $T\gg1$
		\[
		\mathscr{N}(t_0)\ll\frac{1}{M^2}\sum_{h\leq H}\sum_{\abs*{\log\frac{p}{x_h}}<1}\frac{|a_{L_h}(p)|}{p^{\sigma_0}}.
		\]
		
		Finally, assume that $\mathrm{i}(h)=1$ for given $h\leq H$. 
		Then assumption \eqref{PNT1} implies that
		\[
		\sum_{\abs*{\log\frac{p}{x_h}}<1}\frac{|a_{L_h}(p)|}{p^{\sigma_0}}\asymp\frac{(\log T)^{1-\sigma_0}}{M^{1-\sigma_0}\log\log T}\asymp\sum_{\abs*{\log\frac{p}{x_h}}<\frac{1}{2}}\frac{|a_{L_h}(p)|}{p^{\sigma_0}}.
		\]
		If on the other hand $\mathrm{i}(h)=2$, then $L_h\in\mathcal{S}_{\text{poly}}$ and satisfies the Ramanujan Hypothesis.
		Hence, $|a_{L_h}(p)|^2/\partial_{L_h}\leq|a_{L_h}(p)|\leq\partial_{L_h}$ which in combination with \eqref{PNT1} implies that
		\[
		\sum_{\abs*{\log\frac{p}{x_h}}<1}\frac{|a_{L_h}(p)|}{p^{\sigma_0}}\ll\frac{(\log T)^{1-\sigma_0}}{M^{1-\sigma_0}\log\log T}\asymp\sum_{\abs*{\log\frac{p}{x_h}}<\frac{1}{2}}\frac{|a_{L_h}(p)|^2}{p^{\sigma_0}}\ll\sum_{\abs*{\log\frac{p}{x_h}}<\frac{1}{2}}\frac{|a_{L_h}(p)|}{p^{\sigma_0}}.
		\]
		Thus, the above bounds, relation \eqref{Mont} and the inequality $1-2\pi^2\norm{x}^2\leq\cos(2\pi x)$, $x\in\mathbb{R}$, yield that, for any $h\leq H$,
		\begin{align*}
			\max_{|t|\leq\tau}\mathrm{Re}\parentheses{e^{-i\phi_h}\log L_h(s_0+it)}
			&\geq\frac{1}{4}\sum_{\abs*{\log\frac{p}{x_h}}<\frac{1}{2}}\frac{|a_{L_h}(p)|}{p^{\sigma_0}}-\pi^2\mathscr{N}(t_0)+O\parentheses*{\frac{(\log T)^{1-\sigma_0}}{M\log\log T}}\\
			&\gg\parentheses*{1-\frac{1}{M^2}}\frac{(\log T)^{1-\sigma_0}}{M^{1-\sigma_0}\log\log T}+O\parentheses*{\frac{(\log T)^{1-\sigma_0}}{M\log\log T}}.
		\end{align*}
		By fixing $M\gg1$ and taking any sufficiently large  $T\gg_M1$ we can prove now the existence of the desired $t_h$, $h\leq H$, which concludes the proof of the theorem.
	\end{proof}
	\section{Concluding Remarks}\label{concluding}
	\subsection{On the critical line}
As it was mentioned in the introduction, Theorem \ref{thm:ourresult0} is an unconditional version of \eqref{motibvation1} for $\mathrm{GL}(1)$ and $\mathrm{GL}(2)$ $L$-functions when $H>3$ with a lower bound smaller by a power of $1/\sqrt{H-1}$.
This restriction can be traced back at first in Proposition \ref{keyprop} where we require $\lambda>v$ when $q=u/v$.
Since $q$ can be at most $2/(H-1)$ for proving the theorem, $\lambda$ has to be at least $H-1$.
The source of the restriction $\lambda>v$ in Proposition \ref{keyprop} can be found in Lemma \ref{lowerboundapprox2} and it has to do, rather naturally, with the possible existence of nontrivial zeros of the considered $L$-functions.
Lemma  \ref{lowerboundapprox2} provides basically the means to control the average difference between $L_\pi(s)^{q}$ and its short Dirichlet polynomial approximation $s_\pi(s,T)$.
Ideally, we would like to apply Gabriel's convexity theorem (Lemma \ref{Gabriel0}) directly to the function $L_\pi(s)^q-s_\pi(s,T)$ with $y=1$ and translate the problem into the half-plane of absolute convergence.
Unconditionally, however, we do not know whether $L_\pi(s)^q$, $q\notin\mathbb{N}$, is analytic in the half-plane $\sigma>1/2$.
Heath-Brown's idea was to instead apply Gabriel's convexity theorem to $L_\pi(s)^u-s_\pi(s,T)^v$ with $y=1/v$ which suffices for the subsequent computations. 
It is exactly here where $s_\pi(s,T)^v$, though still a good approximation of $L_\pi(s)^u$, does not ``delete'' enough terms from $L_{\pi}(s)^u$ when subtracted.
That is, the coefficients $C_i(n)$, $n\geq1$, as defined above \eqref{D3} should preferably be vanishing for $n\leq X^v$ rather than $n\leq X$.

So let us assume the GRH for all $L_h$, $h\leq H$, and  that none of the $L$-functions has a pole at $s=1$.
In a different case we would have to work a little harder and additionally employ another convexity theorem of Gabriel \cite[Theorem 1]{Gabriel1927} to overcome the possible poles while moving to the half-plane of absolute convergence.
Let 
$
\tilde{d}_\pi(s,T):=L_\pi(s)^{q}-s_\pi(s,T)$ and 
\[ 
\tilde{D}_\pi(\sigma,T):=\frac{1}{T}\int_{\mathbb{R}}\abs*{\tilde{d}_\pi(\sigma+it,T)}\abs*{R(\sigma+it,T)}^2w(t,T)\mathrm{d}t.
\]
For the sake of the succeeding discussion, we also set the length of $s_\pi(\sigma,T)$ to be slightly longer then the one of the resonator, $T^{\tilde{\Delta}}$ say for some $\Delta<\tilde{\Delta}<1/2$.
Applying Lemma \ref{Gabriel0} with $y=1$ to the function $f(s)=\tilde{d}_\pi(s,T)R(s)^{2}\exp\parentheses*{(s-i\tau)^2}$, which satisfies the necessary requirements in the strip $1/2\leq\sigma\leq\beta$,  and arguing exactly as in Lemma \ref{lowerboundapprox2} will result to
	\begin{align}\label{conditional}
	\tilde{D}_\pi(\sigma,T)
	&\ll \tilde{D}_\pi({\textstyle{\frac{1}{2}}},T)^{1+\frac{1/2-\sigma}{\beta-1/2}}T^{(1/2-\sigma)\parentheses*{1-1/\beta}\tilde{\Delta}}+e^{-\epsilon T^{2}}.
\end{align}
Now there is no factor $1/v$ in the power to which $T$ is raised in the right-hand side above.
If we were to use $\tilde{D}_\pi(\sigma,T)$ instead of ${D}_\pi(\sigma,T)$  in the proof of Proposition \ref{keyprop}	and employ \eqref{conditional} instead of Lemma \ref{lowerboundapprox2}, we could relax the restriction $\lambda>v$ to $\lambda>1$ (see in particular \eqref{Heath-Brown} where the factor $1/v$ will be replaced by $1$).
Thus, under the GRH it is possible to increase in Theorem \ref{thm:ourresult0} the power $1/\sqrt{H-1}$ to $1$.
We can pursue this direction further since the above seemingly small detail can aid us to retrieve  \eqref{motibvation1} for an arbitrary number of $\mathrm{GL}(m)$ $L$-functions.
Based on how we argued to prove Theorem \ref{thm:ourresult0}, it suffices to prove that the inequality stated in Proposition \ref{keyprop} can be reversed.
Once more we can modify Heath-Brown's method on upper bounds of fractional moments of $\zeta(s)$ to conditionally obtain the desired inequality whenever $\Delta<2\tilde{\Delta}$ which would result to the more restricted range  $D<1/4$ compared to $D<1/2$ of \eqref{motibvation1} (\cite[Theorem 3]{Heap2024}).
	\subsection{Off the critical line}
There is a rich literature on zero-density estimates since they are the closest we have unconditionally in place of the GRH and often yield results of similar strength.
Their quality can differ depending  on which $L$-function we consider and how far we work from the critical line.
Theorem \ref{thm:ourresult1} and Theorem \ref{thm:ourresult2} require the bare minimum from these estimates, namely to grow essentially by $o(T)$.
At the same time, their speed of growth determines the width of the horizontal strips where we are certain that there are no zeros and in consequence they also determine the size of $D$ in the respective theorems.
	Generally, it is known that any $\mathrm{GL}(1)$ and $\mathrm{GL}(2)$ $L$-function satisfies
	\[
	N_{L}(\sigma, T)\ll T^{1-\alpha(\sigma-1/2)}\log T,\quad\sigma\geq\frac{1}{2},\, T\geq1
	\]
	for some $\alpha>0$.
It is unknown whether higher degree primitive $L$-functions satisfy such zero-density estimates or any other that is $o(T)$ close to the critical line.
Kaczorowski and Perelli \cite{Kaczorowski2003} showed that for if $L\in\mathcal{S}$ satisfies the Ramanujan Hypothesis, then
\[
N_{L}(\sigma, T)\ll T^{(4d_L+12)(1-\sigma)+\epsilon},\quad\sigma\geq\frac{1}{2},\, T\geq1.
\]
Therefore, Theorem \ref{thm:ourresult1} and Theorem \ref{thm:ourresult2} are applicable in the range $\sigma>1-1/(4d_L+12)$.
It is also not very difficult to increase the range to $\sigma>1-1/d_L$ by an appeal to a second moment estimate of $L(s)$ on the critical line, which unconditionally can be shown to be $O(T^{d_L/2+\epsilon})$, and classic results of Hal\'asz and Montgomery (see for example \cite{MUKHOPADHYAY2007}) and \cite{Pintz2024}).
	\appendix
	\section{List of estimates}
	\begin{namedthm*}{Lemma A}\label{resonator properties}
		Let $L_1,\dots, L_H,R$ be as in Section \ref{firsttheorem}, $\eta:=(\log\mathscr{L})^{-3}$ and $\sigma\in[1/2-\eta,1/2+\eta]$.
		We then have that
		\begin{enumerate}[label=(A.\arabic*), wide, labelindent=0pt]
			\item\label{resonator properties(i)} \[r(p)=o(1)=\frac{|a_{L_h}(p)||r(p)|}{{p}^{2\sigma-1/2}}\quad\text{uniformly for }p\in\mathcal{P}\text{ as }T\to+\infty,
			\]
			\item\label{resonator properties(ii)} 
			\[
			\sum_{p}|r(p)|^2=\parentheses*{\frac{1}{2\lambda}+o(1)}\frac{\log X}{\log\log X},
			\]
			\item \label{resonator properties(iii)}\[
			\sum_{p}\frac{\mathrm{Re}\parentheses{\overline{r(p)}a_{L_h}(p)}}{\sqrt{p}(1+|r(p)|^2)}=\parentheses*{{\kappa_{h}}+o(1)}\sqrt{\frac{\log X}{{\kappa_L\lambda}\log\log X}},
			\]
			\item\label{resonator properties(iv)}\[\sum_p\frac{|a_{L_h}(p)||r(p)|}{\sqrt{p}}\ll\sqrt{\frac{\log X\log\log\log X}{\log\log X}},
			\]
			\item\label{resonator properties(v)}
			\begin{align*}
				\sum_p|r(p)|^2\parentheses*{1-p^{1-2\sigma}}=(2\sigma-1)\parentheses*{\frac{1}{\lambda}+o(1)}\log X+O\parentheses*{(2\sigma-1)^2\log X(\log\log X)^2},
			\end{align*}
			\item\label{resonator properties(vi)}
			\begin{align*}
				\sum_{p}\frac{\mathrm{Re}\parentheses{\overline{r(p)}a_{L_h}(p)}}{\sqrt{p}(1+|r(p)|^2)}-\sum_{p}\frac{\mathrm{Re}\parentheses{\overline{r(p)}a_{L_h}(p)}}{p^{2\sigma-1/2}\parentheses*{1+\frac{|r(p)|^2}{p^{2\sigma-1}}}}\ll(2\sigma-1)\sqrt{\log X}(\log\log X)^2.
			\end{align*}
		\end{enumerate}
	\end{namedthm*}
	\begin{proof}
			\ref{resonator properties(i)} follows directly from the construction of $r(n)$.
			
			To prove \ref{resonator properties(ii)} we write
			$
			\sum_{p}|r(p)|^2=M-E,
			$
			where
			\begin{align}\label{Maintermresonator}
				M:=\mathscr{L}^2\sum_{\mathscr{L}^2\leq p\leq\exp((\log \mathscr{L})^2)}\frac{|a_L(p)|^2}{p(\log p)^2}=\parentheses*{\frac{1}{2\lambda}+o(1)}\frac{\log X}{\log\log X}\tag{$\star$},
			\end{align}
			as follows from \eqref{PNTP} and partial summation, and
			\begin{align*}
				E
				&:=\mathscr{L}^2\sum_{\substack{\mathscr{L}^2\leq p\leq\exp((\log \mathscr{L})^2)\\
					\exists h\leq H:\,|a_{L_h}(p)|>(\log p)^{1-\epsilon}}}\frac{|a_L(p)|^2}{p(\log p)^2}\\
				&\ll\mathscr{L}^2\sum_{g,h\leq H}\sum_{\substack{\mathscr{L}^2\leq p\leq\exp((\log \mathscr{L})^2)}}\frac{|a_{L_g}(p)|^2|a_{L_h}(p)|}{p(\log p)^{3-\epsilon}}\\
				&\ll\mathscr{L}^2\sum_{g,h\leq H}\parentheses*{\sum_{p\geq \mathscr{L}^2}\frac{|a_{L_h}(p)|^2}{p(\log p)^{6-2\epsilon}}}^{1/2}\parentheses*{\sum_{p\leq \exp((\log \mathscr{L})^2)}\frac{|a_{L_g}(p)|^4}{p}}^{1/2}.
			\end{align*}
			In view of  \eqref{PNT} and \eqref{fourthmoment} we deduce that 
			$
			E\ll\log X\parentheses*{\log\log X}^{-2+2\epsilon}
			$, which in combination with \eqref{Maintermresonator} yield the desired asymptotics for $\sum_p|r(p)|^2$.
			
			For \ref{resonator properties(iii)} we have that
			\begin{align}\label{twisted}
				\begin{split}
					\sum_{p}\frac{\mathrm{Re}\parentheses{\overline{r(p)}a_{L_h}(p)}}{\sqrt{p}(1+|r(p)|^2)}
					&=\sum_{p}\frac{\mathrm{Re}\parentheses{\overline{r(p)}a_{L_h}(p)}}{\sqrt{p}}+O\parentheses*{\sum_p\frac{|r(p)|^3|a_{L_h}(p)|}{\sqrt{p}}}\\
					&=M_1+E_1-E_2+O\parentheses*{\sum_{p\in\mathcal{P}}\frac{\mathscr{L^3}|a_L(p)||a_{L_h}(p)|}{p^2(\log p)^{1+2\epsilon}}},
				\end{split}
			\tag{$\star\star$}
			\end{align}
			where we employ  \eqref{PNT}, \eqref{fourthmoment} and H\"older's inequality for $E_2$, to show that
			\begin{align*}
				M_1&:=\mathscr{L}\sum_{\mathscr{L}^2\leq p\leq\exp((\log \mathscr{L})^2)}\frac{|a_{L_h}(p)|^2}{p\log p}=\parentheses*{{\kappa_{h}}+o(1)}\sqrt{\frac{\log X}{\kappa_L\lambda\log\log X}},
				\\
				E_1&:=\mathscr{L}\sum_{\substack{g\leq H\\g\neq h}}\sum_{\mathscr{L}^2\leq p\leq\exp((\log \mathscr{L})^2)}\frac{\mathrm{Re}\parentheses*{\overline{a_{L_g}(p)}a_{L_h}(p)}}{p\log p}\ll\frac{\sqrt{\log X}}{(\log\log X)^{3/2}},
				\end{align*}
				and\begin{align*}
				E_2&:=\mathscr{L}\sum_{g\leq H}\sum_{\substack{\mathscr{L}^2\leq p\leq\exp((\log \mathscr{L})^2)\\
						\exists j\leq H:\,|a_{L_j}(p)|>(\log p)^{1-\epsilon}}}\frac{\mathrm{Re}\parentheses{\overline{a_{L_g}(p)}a_{L_h}(p)}}{{p}\log p}\\
				&\ll \mathscr{L}\sum_{g,h,j\leq H}\sum_{\substack{\mathscr{L}^2\leq p\leq\exp((\log \mathscr{L})^2)}}\frac{|a_{L_g}(p)||a_{L_h}(p)||a_{L_j}(p)|}{p(\log p)^{2-\epsilon}}\\
				&\ll \mathscr{L}\parentheses*{\sum_{g\leq H}\sum_{p\leq \exp((\log \mathscr{L})^2)}\frac{|a_{L_g}(p)|^4}{p}}^{3/4}\parentheses*{\sum_{p\geq\mathscr{L}^2}\frac{1}{p(\log p)^{8-4\epsilon}}}^{1/4}\\
				&\ll \frac{\sqrt{\log X}}{(\log\log X)^{3/2-2\epsilon}}.
			\end{align*}
			Lastly, by an application of the Cauchy-Schwarz inequality and \eqref{PNT}, the error term in \eqref{twisted} can be seen to be at most
			\[
			\mathscr{L^3}	\sum_{g\leq H}\sum_{p\geq\mathscr{L}^2}\frac{|a_{L_g}(p)|^2}{p^2\log p}\ll\frac{\sqrt{\log X}}{(\log\log X)^{3/2}}.
			\]
			Gathering all of the above estimates, we derive \ref{resonator properties(iii)}.
			
			\ref{resonator properties(iv)} follows from the Cauchy-Schwarz inequality and relations \eqref{PNT} and \ref{resonator properties(ii)}. 
			Indeed,
			\[
			\sum_p\frac{|a_{L_h}(p)||r(p)|}{\sqrt{p}}\ll\parentheses*{\sum_{p\leq\exp((\log \mathscr{L})^2)}\frac{|a_{L_h}(p)|^2}{p}}^{1/2}\parentheses*{\sum_{p}|r(p)|^2}^{1/2}\ll\sqrt{\frac{\log X\log\log \log X}{\log\log X}}.
			\]
			
			To prove \ref{resonator properties(v)} we observe first that $p^{1-2\sigma}=1+(1-2\sigma)\log p+O\parentheses*{(2\sigma-1)^2(\log p)^2}$, which holds uniformly for $p\in\mathcal{P}$ and $\sigma\in[1/2-\sigma, 1/2+\eta]$ as $T\to+\infty$.
			Therefore,
				\begin{align*}
				\sum_p|r(p)|^2\parentheses*{1-p^{1-2\sigma}}=(2\sigma-1)\sum_{p}|r(p)|^2\log p+O\parentheses*{(2\sigma-1)^2\sum_{p}|r(p)|^2(\log p)^2}
			\end{align*}
			Estimating the sums in the right-hand side as we did for \ref{resonator properties(ii)} yields \ref{resonator properties(v)}.
			
		 For \ref{resonator properties(vi)} we make use again of the formula $p^{1-2\sigma}=1+O((2\sigma-1)\log p)=O(1)$ and \ref{resonator properties(i)} to derive that
			\begin{align*}
				\sum_{p}\frac{\mathrm{Re}\parentheses{\overline{r(p)}a_{L_h}(p)}}{p^{2\sigma-1/2}\parentheses*{1+\frac{|r(p)|^2}{p^{2\sigma-1}}}}
				&=\sum_{p}\frac{\mathrm{Re}\parentheses{\overline{r(p)}a_{L_h}(p)}}{p^{2\sigma-1/2}(1+|r(p)|^2)\parentheses*{1+\frac{|r(p)|^2\parentheses{p^{1-2\sigma}-1}}{1+|r(p)|^2}}}\\
				&=\sum_{p}\frac{\mathrm{Re}\parentheses{\overline{r(p)}a_{L_h}(p)}}{p^{2\sigma-1/2}(1+|r(p)|^2)}+O\parentheses*{(2\sigma-1)\sum_{p}\frac{|r(p)|^3|a_{L_h}(p)|\log p}{\sqrt{p}}},
			\end{align*}
			where the error term is similar to the one in \eqref{twisted} and. hence, it can be estimated in the same way to be at most $\sqrt{\log X/\log\log X}$.
			Moreover,
			\[
			\sum_{p}\frac{\mathrm{Re}\parentheses{\overline{r(p)}a_{L_h}(p)}}{p^{2\sigma-1/2}(1+|r(p)|^2)}=\sum_{p}\frac{\mathrm{Re}\parentheses{\overline{r(p)}a_{L_h}(p)}}{{\sqrt{p}}(1+|r(p)|^2)}+O\parentheses*{(2\sigma-1)\sum_{p}\frac{|r(p)||a_{L_h}(p)|\log p}{\sqrt{p}}},
			\]
			where the sum in the error term can be seen to be $\ll\sqrt{\log X}(\log\log X)^2$ by arguing similarly to how we did in \ref{resonator properties(iv)}. 
	\end{proof}
	\bibliographystyle{00-plainnat}
	\bibliography{joint-extreeme}
\end{document}